\documentclass[twoside, a4paper, 10pt
]{amsart}

\usepackage{amsmath}  
\usepackage{amsfonts}
\usepackage{amsthm}  
\usepackage{amssymb}
\usepackage{enumerate}
\usepackage[all]{xy}
\SelectTips{cm}{}
\SilentMatrices
\usepackage[pdfauthor={Steffen Sagave}, pdftitle={Logarithmic structures on topological K-theory spectra}]{hyperref}

\numberwithin{equation}{section}
\theoremstyle{plain}
\newtheorem{lemma}[subsection]{Lemma}
\newtheorem{theorem}[subsection]{Theorem}

\newtheorem{corollary}[subsection]{Corollary}
\newtheorem{proposition}[subsection]{Proposition}
\theoremstyle{definition}

\newtheorem{construction}[subsection]{Construction}
\newtheorem{definition}[subsection]{Definition}
\newtheorem{example}[subsection]{Example}
\newtheorem{remark}[subsection]{Remark}

\newcommand{\mN}{{\mathbb N}}

\newcommand{\mS}{{\mathbb S}}
\newcommand{\mZ}{{\mathbb Z}}

\newcommand{\cC}{{\mathcal C}}

\newcommand{\cI}{{\mathcal I}}
\newcommand{\cJ}{{\mathcal J}}

\newcommand{\cN}{{\mathcal N}}

\newcommand{\cP}{{\mathcal P}}
\newcommand{\cS}{{\mathcal S}}

\DeclareMathOperator{\id}{id}

\DeclareMathOperator{\Mod}{Mod}
\DeclareMathOperator{\Map}{Map}
\DeclareMathOperator{\Der}{Der}

\DeclareMathOperator{\colim}{colim}
\DeclareMathOperator{\hocolim}{hocolim}

\DeclareMathOperator{\THH}{THH}

\DeclareMathOperator{\TAQ}{TAQ}
\DeclareMathOperator{\concat}{\sqcup}
\DeclareMathOperator{\diag}{diag}

\newcommand{\tensor}{\otimes}

\newcommand{\ot}{\leftarrow}
\newcommand{\iso}{\cong}

\newcommand{\op}{{\mathrm{op}}}
\newcommand{\sm}{\wedge}

\newcommand{\wdg}{\vee}
\newcommand{\wdgfib}{\wdg_{\!f\!}}
\newcommand{\Spsym}{{\mathrm{\textrm{Sp}^{\Sigma}}}}

\newcommand{\bld}[1]{{\mathbf{#1}}}

\DeclareMathOperator{\capitalGL}{GL}

\newcommand{\bof}[1]{b{#1}}

\newcommand{\GLoneI}{\capitalGL^{\cI}_1}
\newcommand{\GLoneIof}[1]{\capitalGL^{\cI}_1\!{#1}}
\newcommand{\GLoneJof}[1]{\capitalGL^{\cJ}_1\!\!{#1}}

\newcommand{\GammaS}{\Gamma^{\op}\!\text{-}\cS}

\newcommand{\pre}{\mathrm{pre}}

\newcommand{\cof}{\mathrm{cof}}
\newcommand{\fib}{\mathrm{fib}}
\newcommand{\gp}{\mathrm{gp}}
\newcommand{\dlog}{\,\mathrm{d}\mathit{log}\,}

\newcommand{\arxivlink}[1]{\href{http://arxiv.org/abs/#1}{\texttt{arXiv:#1}}}

\usepackage[pdftex]{graphics,color} 
\usepackage{eso-pic}
\usepackage{scrtime}
\usepackage{ifdraft}
\ifdraft{%
\AddToShipoutPicture{\put(30,30){\resizebox{!}{.4cm}%
     {{{\color[gray]{0.8}\texttt{Preliminary draft (version of \the\year-\the\month-\the\day)}}}}}}
}{}

\begin{document} 
\title[Logarithmic structures on topological \texorpdfstring{$K$}{K}-theory spectra]{Logarithmic structures \\ on topological \texorpdfstring{$K$}{K}-theory spectra}

\author{Steffen Sagave} \address{Mathematical
Institute, University of Bonn, Endenicher Allee 60, 53115 Bonn,
Germany} \email{sagave@math.uni-bonn.de}

\date{\today}

\begin{abstract}
  We study a modified version of Rognes' logarithmic structures on
  structured ring spectra. In our setup, we obtain canonical
  logarithmic structures on connective $K$-theory spectra which
  approximate the respective periodic spectra. The inclusion of the
  $p$-complete Adams summand into the $p$-complete connective complex
  $K$-theory spectrum is compatible with these logarithmic
  structures. The vanishing of appropriate logarithmic
  \mbox{topological} Andr\'{e}-Quillen homology groups confirms that
  the inclusion of the Adams \mbox{summand} should be viewed as a
  tamely ramified extension of ring spectra.
\end{abstract}
\maketitle

\section{Introduction} 
A \emph{pre-log structure} on a commutative ring $A$ is a commutative
monoid $M$ together with a monoid map $\alpha \colon M \to (A,\cdot)$
into the multiplicative monoid of $A$. It is a \emph{log structure} if
the map $\alpha^{-1}(A^{\times}) \to A^{\times}$ from the submonoid
$\alpha^{-1}(A^{\times}) \subseteq M$ of elements mapping to the units
$A^{\times}$ of $A$ is an isomorphism. The trivial log structure
$A^{\times} \hookrightarrow (A,\cdot)$ is the easiest example.  A
\emph{log ring} is a commutative ring with a log structure.  This
notion is the affine version of the \emph{log schemes} studied in
algebraic geometry. Log schemes are useful because they for example
enlarge the range of smooth and \'{e}tale maps.

We consider a very basic example of interest to us: An integral domain
$A$ may be viewed as a log ring $(A,M)$ with $M = A\setminus 0$. The
localization map to its fraction field $A \to K = A[M^{-1}]$  admits a
factorization
\begin{equation}\label{eq:AtoK-factorization}
(A,A^{\times}) \to (A,M) \to (K,K^{\times})
\end{equation} 
in log rings, and we may view $(A,M)$ as an intermediate localization
of $A$ which is ``milder'' than $K$.  In contrast, $A \to K$ does in
general not factor in a non-trivial way as a map of commutative rings.

We switch to the topological $K$-theory spectra of algebraic
topology. Consider the inclusion $\ell_p \to ku_p$ of the $p$-complete
Adams summand into the $p$-complete connective complex $K$-theory
spectrum $ku_p$. On homotopy groups, it induces a map $\mZ_{p}[v_1]
\to \mZ_{p}[u]$ sending $v_1$ to $u^{p-1}$. Since $p-1$ is invertible
in $\mZ_{p}$, this map of homotopy groups behaves like a tamely
ramified extension if we interpret $u$ and $v_1$ as uniformizers.  As
observed by Hesselholt and explained by Ausoni~\cite[\S
10.4]{Ausoni_THH_ku}, computations of the topological Hochschild
homology $\THH$ of $ku_p$ and $\ell_p$ provide a much deeper reason
for why $\ell_p \to ku_p$ should be viewed as a tamely ramified
extension on the level of structured ring spectra: On certain relative
$\THH$ terms, $\ell_p \to ku_p$ shows the same behavior as tamely
ramified extensions of discrete valuation rings whose $\THH$ is
studied by Hesselholt-Madsen~\cite[\S 2]{Hesselholt-M_local_fields}.
This is also supported by the $\THH$ localization sequences for
$\ell_p$ and $ku_p$ established by
Blumberg-Mandell~\cite{Blumberg-M_loc-sequenceTHH}.

In order to explain this and other phenomena arising in connection
with the $\THH$ and algebraic $K$-theory of structured ring spectra,
Rognes introduced a notion of \emph{log ring
  spectra}~\cite{Rognes_TLS}. This is a homotopical generalization of
the log rings defined above in which ring spectra play the role of
commutative rings. One question that remained open in Rognes' work
was how to extend $\ell_p \to ku_p$ to a \emph{formally log \'{e}tale}
map of log ring spectra, i.e, a map whose \emph{log topological
  Andr\'{e}-Quillen homology} vanishes. In the algebraic setup, the
vanishing of the corresponding module of \emph{log K\"ahler
  differentials} detects tame ramification. So being formally log
\'{e}tale is one reasonable candidate for a definition of a
tamely ramified extension of ring spectra, and $\ell_p \to
ku_p$ should be formally log \'{e}tale with respect to suitable log
structures.

The aim of the present paper is to resolve the above issue by
modifying Rognes' definition to what we call \emph{graded} log ring
spectra. There are canonical graded log structures on connective
$K$-theory spectra like $ku_{p}$ and $\ell_p$ turning them into graded
log ring spectra. Generalizing the algebraic
example~\eqref{eq:AtoK-factorization}, the latter provide intermediate
objects between connective and periodic $K$-theory spectra equipped
with their trivial graded log structures. The localizations of these
intermediate graded log ring spectra are the respective periodic
$K$-theory spectra. Moreover, $\ell_p \to ku_p$ extends to a map of
graded log ring spectra which is \emph{formally graded log \'{e}tale},
that is, a map whose \emph{graded log topological Andr\'{e}-Quillen
  homology} vanishes.

\subsection{Logarithmic ring spectra}
Structured ring spectra provide a homotopical generalization of
commutative rings which allows one to transfer many concepts from algebra
to homotopy theory, including for example algebraic $K$-theory, Galois
theory, and Morita theory. There are several equivalent definitions of
structured ring spectra. In the present paper, we will work with the
category of \emph{commutative symmetric ring spectra}
$\cC\Spsym$~\cite{HSS,MMSS,Schwede_SymSp}. The objects are symmetric
spectra which are commutative monoids with respect to the smash
product of symmetric spectra.

To generalize the notions of pre-log and log structures, we need to
know what the ``underlying multiplicative monoid'' of a commutative
symmetric ring spectrum $A$ is. If we were working in the more
classical framework of $E_{\infty}$ ring spectra, this would be the
underlying multiplicative $E_{\infty}$ space of an $E_{\infty}$
spectrum. When dealing with (strictly) commutative symmetric ring
spectra, it is more useful to model $E_{\infty}$ spaces by
\emph{commutative $\cI$-space monoids}~\cite{Sagave-S_diagram}. By
definition, a commutative $\cI$-space monoid $M$ is a space valued
functor on the category of finite sets and injections $\cI$ together
with appropriate multiplication maps. These multiplications turn $M$
into a commutative monoid with respect to a convolution product on
this functor category. The category of commutative $\cI$-space monoids
$\cC\cS^{\cI}$ admits a model structure making it Quillen equivalent
to the category of $E_{\infty}$ spaces. Moreover, there is a Quillen
adjunction
\begin{equation}
\mS^{\cI}[-] \colon \cC\cS^{\cI} \rightleftarrows \cC\Spsym \colon \Omega^{\cI}
\end{equation}
whose right adjoint models the underlying multiplicative monoid.

An \emph{$\cI$-space pre-log structure} on a commutative symmetric
ring spectrum $A$ is then a commutative $\cI$-space monoid $M$
together with a map $\alpha \colon M \to \Omega^{\cI}(A)$ in
$\cC\cS^{\cI}$~\cite[Definition 7.1]{Rognes_TLS}. The units of $A$ are
the subobject $\GLoneIof(A)$ of $\Omega^{\cI}(A)$ given by the
invertible path components. In analogy with the algebraic definition,
an $\cI$-space pre-log structure $(M,\alpha)$ is an \emph{$\cI$-space
  log structure} if $\alpha^{-1}(\GLoneIof(A)) \to
\GLoneIof(A)$ is a weak equivalence of commutative $\cI$-space monoids.

Although this is an obvious and useful generalization of the algebraic
definition, we would like to emphasize one aspect which is not
optimal: If $i\colon A \to B$ is a map of commutative rings, then the
pullback of
\begin{equation}\label{eq:direct-image-of-trivial-alg} 
  B^{\times} \to (B,\cdot) \ot (A,\cdot) \end{equation}
defines a log structure $i_*(B^{\times})$ on $A$. For example, the log
ring $(A,M)$ in~\eqref{eq:AtoK-factorization} arises from $A\to K$ in
this way.  Let $ku$ be the connective complex $K$-theory
spectrum. Since the Bott class $u$ becomes invertible in $\pi_*(KU)$,
the periodic $KU$ has more units in its homotopy groups, and one may
hope that the map $i\colon ku \to KU$ into the periodic spectrum
induces an interesting $\cI$-space log structure as
in~\eqref{eq:direct-image-of-trivial-alg}. However, the pullback of
\begin{equation}\label{eq:direct-image-of-trivial-I}
\GLoneIof(KU) \to \Omega^{\cI}(KU) \ot \Omega^{\cI}(ku)
\end{equation}
only provides the trivial log structure on $ku$. The problem is that
$\GLoneIof(ku)$ and $\GLoneIof(KU)$ are equivalent: $\GLoneI(KU)$ only
detects the units in $\pi_0(KU)\iso \mZ$ and ignores that the graded
ring $\pi_*(KU)=\mZ[u^{\pm 1}]$ has more units than
$\pi_*(ku)=\mZ[u]$.

\subsection{Graded \texorpdfstring{$E_{\infty}$}{E-infinity} spaces}
To overcome the difficulty outlined in the previous example, it is
desirable to have a notion of the units of a ring spectrum $A$ which
takes all units in the graded ring $\pi_*(A)$ into account. Such
\emph{graded} units have been defined by the author in joint work with
Schlichtkrull~\cite{Sagave-S_diagram}.

The key idea behind the graded units is to replace the category $\cI$
used above by a more elaborate indexing category. The appropriate
choice turns out to be the category $\cJ=\Sigma^{-1}\Sigma$ given by
Quillen's localization construction on the category of finite sets and
bijections $\Sigma$. This $\cJ$ is a symmetric monoidal category whose
classifying space $B\cJ$ has the homotopy type of $QS^0$. The same
constructions as in the case of $\cI$-spaces lead to a model category
of commutative $\cJ$-space monoids $\cC\cS^{\cJ}$. We show
in~\cite{Sagave-S_diagram} that $\cC\cS^{\cJ}$ is Quillen equivalent
to the category of $E_{\infty}$ spaces over $B\cJ$. So commutative
$\cJ$-space monoids correspond to ``commutative monoids over the
underlying additive monoid of the sphere spectrum'', just as
$\mZ$-graded monoids in algebra can be defined as commutative monoids
over the additive monoid of $\mZ$. This is why we think of commutative
$\cJ$-space monoids as graded $E_{\infty}$ spaces.

The reason for why $\cJ$ is useful for studying units is a beautiful
connection to the combinatorics of symmetric spectra~\cite[\S
4.21]{Sagave-S_diagram}. It gives rise to a Quillen adjunction
\begin{equation}\label{eq:OmegaJ-adj-intro}
\mS^{\cJ}[-] \colon \cC\cS^{\cJ} \rightleftarrows \cC\Spsym \colon \Omega^{\cJ}.
\end{equation}
If $A$ is a commutative symmetric ring spectrum, we think of
$\Omega^{\cJ}(A)$ as the underlying graded multiplicative monoid of
$A$. The point about $\Omega^{\cJ}(A)$ is that it is built from all
spaces $\Omega^{m_2}(A_{m_1})$, while the $\cI$-space version only
uses the spaces $\Omega^{m}(A_{m})$.  This makes it possible to define
a commutative $\cJ$-space monoid $\GLoneJof(A) \subseteq
\Omega^{\cJ}(A)$ of \emph{graded units} of $A$ from which we can
recover all units in the graded ring $\pi_*(A)$.

\subsection{Graded logarithmic ring spectra}
Given the previous discussion, we define a \emph{graded pre-log
  structure} on a commutative symmetric ring spectrum $A$ to be a
commutative $\cJ$-space monoid $M$ together with a map $\alpha\colon M
\to \Omega^{\cJ}(A)$. It is a log structure if
$\alpha^{-1}(\GLoneJof(A)) \to \GLoneJof(A)$ is a weak equivalence in
$\cC\cS^{\cJ}$. The pullback
\begin{equation}\label{eq:direct-image-of-trivial-J}
\GLoneJof(KU) \to \Omega^{\cJ}(KU) \ot \Omega^{\cJ}(ku)
\end{equation}
provides an interesting non-trivial graded log structure $i_*\!\GLoneJof(KU)$
on $ku$. In analogy with the algebraic
situation~\eqref{eq:AtoK-factorization}, 
the $(ku,i_*\!\GLoneJof(KU))$ is part of a factorization
\begin{equation}\label{eq:ku-KU-factorization}
(ku,\GLoneJof(ku) ) \to (ku,i_*\!\GLoneJof(KU)) \to (KU,\GLoneJof(KU)).
\end{equation}

Let $(M,\alpha)$ be a graded log structure on
$A$. In~\cite{Sagave_spectra-of-units}, we construct a group
completion $M \to M^{\gp}$ for commutative $\cJ$-space
monoids. Together with the adjoint of $\alpha\colon M \to
\Omega^{\cJ}(A)$ under the adjunction~\eqref{eq:OmegaJ-adj-intro}, the
group completion induces maps
\[\mS^{\cJ}[M^{\gp}] \ot \mS^{\cJ}[M]\to A\]
in $\cC\Spsym$. The \emph{localization} $A[M^{-1}]$ of $(A,M)$ is the
pushout of this diagram in commutative symmetric ring spectra. As one
may hope from the algebraic example, the next theorem allows us to
interpret $(ku,i_*\!\GLoneJof(KU))$ as an approximation to $KU$. It is
the special case of a more general theorem about log structures on
connective covers of periodic ring spectra.

\begin{theorem}\label{thm:log-ell-ku-introduction}
  The map $(ku, i_*\!\GLoneJof(KU)) \to (KU, \GLoneJof(KU))$ induces a
  stable equivalence between the localization
  $ku[(i_*\!\GLoneJof(KU))^{-1}]$ and the periodic spectrum~$KU$. A
  similar statement holds for the $p$-complete and $p$-local
  connective complex $K$-theory spectra $ku_p$ and $ku_{(p)}$ and their
  Adams summands $\ell_p$ and $\ell$.
\end{theorem}

The proof of the theorem depends heavily on our analysis of the group
completion for commutative $\cJ$-space monoids
in~\cite{Sagave_spectra-of-units}. The various $\cI$-space and
operadic pre-log and log structures considered by
Rognes~\cite{Rognes_TLS} do not have this property.

\subsection{Formally log \'{e}tale extensions} Let $(f,f^{\flat})
\colon (R,P,\rho) \to (A,M,\alpha)$ be a map of log rings in the
algebraic setup. If $X$ is an $A$-module, a log derivation of
$(f,f^{\flat})$ with values in $X$ is a pair $(D,\delta)$ where
$D\colon A \to X$ is an ordinary $R$-linear derivation of $A$ and
$\delta\colon M \to X$ is a monoid map such that $\delta f^{\flat} =
0$ and $D(\alpha(m)) = \alpha(m)
\delta(m)$~\cite{Kato_logarithmic-structures, Ogus_log-book}. The
resulting set of log derivations $\Der_{(R,P)}((A,M), X)$ is
corepresented by the $A$-module of \emph{log K\"ahler differentials}
$\Omega^1_{(A,M)/(R,P)}$.  This module is isomorphic to the quotient
of $\Omega^1_{A/R} \oplus (A\otimes M^{\gp})$ obtained by imposing the
relations $d(\alpha(m)) = \alpha(m)\tensor m$ and $1\tensor
f^{\flat}(p) = 0$ for monoid elements $m\in M$ and $p\in P$. Writing
$a\dlog m$ for $a \tensor m$, the first relation shows that $\dlog m$
has the properties of a logarithmic differential. This is the source
of the term ``log'' in this theory.  The relation also shows that
generators of the form $d\alpha(m)$ in $\Omega^1_{A/R}$ become once
divisible by $\alpha(m)$ when passing to $\Omega^1_{(A,M)/(R,P)}$. So
it is a milder localization of $\Omega^1_{A/R}$ than
$\Omega^1_{A[M^{-1}]/R}$.

In the context of ring spectra, Basterra~\cite{Basterra_Andre-Quillen}
has shown that $R$-linear derivations of $A$ are corepresented by the
\emph{topological Andr\'{e}-Quillen homology} $\TAQ^R(A)$. Building on
Basterra's result, Rognes has shown that log derivations of log ring
spectra are corepresented by a \emph{log topological Andr\'{e}-Quillen
  homology}. In the present paper, we construct the corresponding
\emph{graded log topological Andr\'{e}-Quillen homology}
$\TAQ^{(R,P)}(A,M)$. It corepresents graded log derivations. The
definition and the analysis of the graded log $\TAQ$ rely on the
equivalence between the homotopy category of grouplike commutative
$\cJ$-space monoids and a suitable homotopy category of augmented
connective spectra established in~\cite{Sagave_spectra-of-units}.

A map $(R,P) \to (A,M)$ of graded log ring spectra is \emph{formally
  graded log \'{e}tale} if the graded log topological
Andr\'{e}-Quillen homology $\TAQ^{(R,P)}(A,M)$ is contractible.
\begin{theorem}\label{thm:ell-to-k-formally-etale-intro}
The inclusion $\ell_p \to ku_p$ extends to map of graded log ring spectra
\[(\ell_p, i_*\!\GLoneJof(L_p)) \to (ku_{p}, i_*\!\GLoneJof(KU_{p})),\]
and this map is formally graded log \'{e}tale. The same hold in the $p$-local case. 
\end{theorem}
For the proof, we replace these log structures by ``smaller'' pre-log
structures and analyze them using group completions.  As with
Theorem~\ref{thm:log-ell-ku-introduction}, the $\cI$-space and
operadic log structures on $\ell_p$ and $ku_p$ considered by Rognes do
not have this property.

It is also possible to transfer Rognes' notion of \emph{logarithmic
  topological Hochschild homology} to our graded context. Applying
this to the graded log structures discussed above provides interesting
homotopy cofiber sequences relating logarithmic and ordinary
$\THH$. This is studied in detail in~\cite{Rognes-S-S_logTHH}.

\subsection{Organization}
In Section~\ref{sec:spsym-csj-review} we give background about
symmetric spectra and commutative $\cJ$-space monoids.
Section~\ref{sec:graded-log-ring-spectra} contains basic terminology
about graded log structures. In Section~\ref{sec:direct-image-log} we
study various pre-log and log structures generated by homotopy classes
and obtain Theorem~\ref{thm:log-ell-ku-introduction} as a special case
of Theorem~\ref{thm:all-pre-log-on-A} proven there. The graded log
topological Andr\'{e}-Quillen homology is constructed in
Section~\ref{sec:log-taq}.  Section ~\ref{sec:log-etale-maps} features
the proof of Theorem~\ref{thm:ell-to-k-formally-etale-intro}. The
final Section~\ref{sec:CSJ-derivations} contains auxiliary results
about commutative $\cJ$-space monoids.

\subsection{Acknowledgments}
The author thanks John Rognes for introducing him to topological
logarithmic structures, for discussions about the present work, and
for comments on an earlier version of this manuscript. Moreover, the author
likes to thank Christian Schlichtkrull for pointing out a mistake
in an earlier version of this paper.

\section{Symmetric ring spectra and structured diagram
  spaces}\label{sec:spsym-csj-review}
The category of symmetric spectra $\Spsym$ ~\cite{HSS} is a stable
model category whose homotopy category is equivalent to the stable
homotopy category of algebraic topology. The smash product of
symmetric spectra $\sm$ induces the smash product on the homotopy
category. A \emph{commutative symmetric ring spectrum} is a
commutative monoid in $(\Spsym, \sm)$. We write $\cC\Spsym$ for the
category of commutative symmetric ring spectra. It admits a
\emph{positive stable model structure}~\cite{MMSS} making it Quillen
equivalent to the category of $E_{\infty}$ spectra. So all homotopy types
of $E_{\infty}$ spectra are represented by commutative symmetric ring
spectra, and in many cases it is even possible to write down explicit
models in $\cC\Spsym$~\cite{Schwede_SymSp}.

\subsection{Commutative \texorpdfstring{$\cI$}{I}-space monoids}
We will now recall from~\cite[\S 3]{Sagave-S_diagram} how one can use
structured diagram spaces to model the ``underlying multiplicative
$E_{\infty}$ spaces'' of commutative symmetric ring spectra.
\begin{definition}
  Let $\cI$ be the category with objects the sets
  $\bld{m}=\{1,\dots,m\}$ for $m \geq 0$ and morphisms the injective
  maps. The ordered concatenation $\concat$ of ordered sets turns
  $\cI$ into a symmetric monoidal category with strict unit and
  associativity. The monoidal unit is the empty set $\bld{0}$. The
  symmetry isomorphism is the shuffle $\chi_{m,n}\colon
  \bld{m}\concat\bld{n} \to \bld{n}\concat \bld{m}$ moving the first
  $m$ elements past the last $n$ elements.
\end{definition}

An \emph{$\cI$-space} is a functor from $\cI$ to the category of
unpointed simplicial sets~$\cS$. The resulting functor category
$\cS^{\cI}$ of $\cI$-spaces inherits a symmetric monoidal product
$\boxtimes$ from $\cI$ and $\cS$: For $\cI$-spaces $X$ and $Y$,
the product $X \boxtimes Y$ is the left Kan extension of their object-wise
cartesian product along $- \concat - \colon \cI \times \cI \to \cI$. So
\[ (X\boxtimes Y)(\bld{n}) = \colim_{\bld{k}\concat\bld{l} \to
  \bld{n}} X(\bld{k})\times Y(\bld{l}).\] The monoidal unit for
$\boxtimes$ is the $\cI$-space $U^{\cI}=\cI(\bld{0},-)$.

\begin{definition}
  A \emph{commutative $\cI$-space monoid} is a commutative monoid in
  $(\cS^{\cI},\boxtimes, U^{\cI})$, and $\cC\cS^{\cI}$ denotes the
  category of commutative $\cI$-space monoids.
\end{definition}
More explicitly, a commutative $\cI$-space monoid $M$ is an
$\cI$-space $M$ together with multiplications $M(\bld{k})\times
M(\bld{l}) \to M(\bld{k}\concat\bld{l})$ and a unit map $* \to
M(\bld{0})$ satisfying appropriate coherence conditions.

While strictly commutative simplicial monoids fail to model all
homotopy types of $E_{\infty}$ spaces, the additional symmetry of
$\cI$-spaces and the use of a \emph{positive} model structure ensure
that $E_{\infty}$ spaces admit strictly commutative models in
$\cI$-spaces:
\begin{theorem}\cite[Theorem 1.2]{Sagave-S_diagram}
  The category of commutative $\cI$-space monoids $\cC\cS^{\cI}$
  admits a positive $\cI$-model structure making it Quillen equivalent
  to the category of $E_{\infty}$ spaces.
\end{theorem}
The weak equivalences in this \emph{positive $\cI$-model structure}
are the \emph{$\cI$-equivalences}, i.e., the maps $f\colon M \to N$
which induce weak equivalences $f_{h\cI}\colon M_{h\cI} \to N_{h\cI}$
on homotopy colimits. Here
\[ M_{h\cI}=\hocolim_{\cI}M = \diag\left( [s] \mapsto
  \displaystyle\coprod_{\bld{k_0} \ot \dots \ot \bld{k_s}}
  M(\bld{k_s})\right)\] is the usual Bousfield-Kan homotopy colimit.
A commutative $\cI$-space monoid $A$ is fibrant in this model
structure if every morphism $\bld{k} \to \bld{l}$ in $\cI$ with $k\geq
1$ induces a weak equivalence of Kan-complexes $M(\bld{k}) \to
M(\bld{l})$. The positivity condition $k\geq 1$ ensures that we do not
represent the common homotopy type of the $M(\bld{k})$ by a
commutative simplicial monoid.

Commutative $\cI$-space monoids are relevant in connection with
symmetric ring spectra because there is a Quillen adjunction
\[
\mS^{\cI}[-] \colon \cC\cS^{\cI} \rightleftarrows \cC\Spsym \colon \Omega^{\cI}
\]
with respect to the positive model structures~\cite[Proposition
3.19]{Sagave-S_diagram}. The right adjoint $\Omega^{\cI}$ is on
objects given by $\Omega^{\cI}(A)(\bld{m})=\Omega^m(A_m)$, and the
multiplication $\Omega^{k}(A_k)\times\Omega^{l}(A_l) \to
\Omega^{k+l}(A_{k+l})$ sends $(f,g)$ to the composite
\begin{equation}\label{eq:mult-on-OmegaI} S^{k+l}\iso S^{k} \sm S^{l} \xrightarrow{f\sm g} A_{k} \sm A_{l}
  \to A_{k+l}.\end{equation}
For a positive fibrant $A$ in $\cC\Spsym$, the $\Omega^{\cI}(A)$
models the underlying multiplicative $E_{\infty}$ space of $A$. The
\emph{units} $\GLoneIof(A)$ of $A$ is the sub commutative $\cI$-space monoid of
$\Omega^{\cI}(A)$ consisting of those path components that map to
units in the commutative monoid $\pi_0(\Omega^{\cI}(A)_{h\cI}) \iso
\pi_0(A)$.

\subsection{Commutative \texorpdfstring{$\cJ$}{J}-space monoids}
As explained in the introduction, the units $\GLoneIof(A)$ and other
equivalent approaches used in the literature have the undesirable
feature that they do not detect the difference between a periodic
spectrum and its connective cover. We now recall from~\cite[\S
4]{Sagave-S_diagram} how we can overcome this by using a more subtle
indexing category for structured diagram spaces.
\begin{definition}
  Let $\cJ$ be the category whose objects $(\bld{m_1},\bld{m_2})$ are
  pairs of objects in $\cI$. There are no morphisms
  $(\bld{m_1},\bld{m_2})\to (\bld{n_1},\bld{n_2})$ unless $m_2-m_1 =
  n_2 -n_1$, and in this case a morphism $(\bld{m_1},\bld{m_2})\to
  (\bld{n_1},\bld{n_2})$ is a triple $(\beta_1, \beta_2, \sigma)$ with
  the $\beta_i\colon \bld{m_i}\to\bld{n_i}$ injections and $\sigma
  \colon \bld{n_1}\setminus\beta_1(\bld{m_1}) \to
  \bld{n_2}\setminus\beta_2(\bld{m_2})$ a bijection identifying the
  complements of $\beta_1$ and $\beta_2$. The composition of
  \[ (\bld{l_1},\bld{l_2}) \xrightarrow{(\alpha_1,\alpha_2,\rho)}
  (\bld{m_1},\bld{m_2})\xrightarrow{(\beta_1, \beta_2, \sigma)}
  (\bld{n_1},\bld{n_2}) \] is $\beta_i\alpha_i$ in the first two
  entries and the map $\sigma \cup \beta_2\rho\beta_1^{-1}$ in the
  third entry.
\end{definition}
This category $\cJ$ is equivalent to Quillen's localization
construction $\Sigma^{-1}\Sigma$ on the category of finite sets and
bijections $\Sigma$. By the Barratt-Priddy-Quillen theorem (see
e.g.~\cite{Segal_categories}), the classifying space $B\cJ$ has the
homotopy type of $QS^0\simeq \Omega^{\infty}\Sigma^{\infty}S^0$.

Concatenation in both entries makes $\cJ$ a symmetric monoidal
category with monoidal unit $(\bld{0},\bld{0})$. As in the case of
$\cI$-spaces, we obtain a symmetric monoidal category of $\cJ$-spaces
$(\cS^{\cJ},\boxtimes, U^{\cJ})$ with unit
$U^{\cJ}=\cJ((\bld{0},\bld{0}),-)$.
\begin{definition}
  A \emph{commutative $\cJ$-space monoid} is a commutative monoid in
  $(\cS^{\cJ},\boxtimes, U^{\cJ})$, and $\cC\cS^{\cJ}$ denotes the
  category of commutative $\cJ$-space monoids.
\end{definition}
One can define a \emph{positive $\cJ$-model structure} on
$\cC\cS^{\cJ}$ in which the weak equivalences are the
$\cJ$-equivalences, i.e., the maps $f\colon M \to N$ inducing weak
equivalences $f_{h\cJ}\colon M_{h\cJ}\to N_{h\cJ}$ on homotopy
colimits over $\cJ$~\cite[Proposition 4.10]{Sagave-S_diagram}. The
\emph{positive $\cJ$-fibrant} objects are those $M$ for which all
morphisms $(\bld{m_1},\bld{m_2})\to (\bld{n_1},\bld{n_2})$ with
$m_1\geq 1$ induce weak equivalences of Kan complexes
$M(\bld{m_1},\bld{m_2}) \to M(\bld{n_1},\bld{n_2})$. We will
frequently use that this model structure is both left and right proper.
\begin{theorem}\cite[Theorem 1.7]{Sagave-S_diagram}
  With respect to the positive $\cJ$-model structure, $\cC\cS^{\cJ}$
  is Quillen equivalent to the category of $E_{\infty}$-spaces over
  $B\cJ$.
\end{theorem}
As explained in the introduction, this allows us to interpret
commutative $\cJ$-space monoids as \emph{graded $E_{\infty}$ spaces}. The
point in using this specific category $\cJ$ is that there is a Quillen adjunction
\begin{equation}\label{eq:CSJ-CSpsym-adjunction}
\mS^{\cJ}[-] \colon \cC\cS^{\cJ} \rightleftarrows \cC\Spsym \colon \Omega^{\cJ}
\end{equation}
with respect to the positive model structures. On objects, the right
adjoint is given by $\Omega^{\cJ}(A)(\bld{m_1},\bld{m_2}) =
\Omega^{m_2}(A_{m_1})$, and the multiplication on $\Omega^{\cJ}(A)$ is
defined similarly as in~\eqref{eq:mult-on-OmegaI}.  The structure maps
of $\Omega^{\cJ}(A)$ depend on the bijections in the definition of the
morphisms in $\cJ$, see~\cite[\S 4.21]{Sagave-S_diagram} and~\cite[\S
2.9]{Sagave_spectra-of-units}.

For a positive fibrant $A$ in $\cC\Spsym$, we therefore view
$\Omega^{\cJ}(A)$ as the \emph{underlying graded multiplicative
  $E_{\infty}$ space} of $A$. This terminology is justified as
follows: Every point $f\colon S^{m_2}\to A_{m_1}$ in
$\Omega^{\cJ}(A)(\bld{m_1},\bld{m_2})$ represents a homotopy class in
$\pi_{m_2-m_1}(A)$, and one can recover the underlying graded
multiplicative monoid of $\pi_*(A)$ together with its sign action of
$\{\pm 1\}$ from $\Omega^{\cJ}(A)$~\cite[\S
4.21]{Sagave-S_diagram}. (The sign action is inherent to commutative
$\cJ$-space monoids because $\pi_1(B\cJ) \iso \pi_1(\mS) \iso \mZ/2$.)
\begin{definition}
  The \emph{graded units} $\GLoneJof(A)$ of a positive fibrant
  commutative symmetric ring spectrum $A$ is the sub commutative
  $\cJ$-space monoid of $\Omega^{\cJ}(A)$ given by the path components
  that map to units in the graded ring $\pi_*(A)$.
\end{definition}
We collect some results about $\cJ$-spaces needed below. One of the
benefits of working with strictly commutative monoids in $\cS^{\cI}$
or $\cS^{\cJ}$ rather than with (augmented) $E_{\infty}$ spaces is
that coproducts and pushouts admit an explicit construction, similarly
as in commutative rings or commutative ring spectra: The coproduct of
$M$ and $N$ in $\cC\cS^{\cJ}$ is $M\boxtimes N$, and the pushout of $M
\ot P \to N$ is the coequalizer $M \boxtimes_{P} N$ of $M\boxtimes P
\boxtimes N \rightrightarrows M \boxtimes N$.

We will often have to ensure that the coproduct $M\boxtimes N$ is
homotopy invariant in a suitable sense. For this, a cofibrancy
condition is necessary.
\begin{definition}
  Let $\partial(\cJ\downarrow(\bld{n_1},\bld{n_2}))$ be the full
  subcategory on the objects in $\cJ\downarrow(\bld{n_1},\bld{n_2})$
  which are not isomorphisms. A $\cJ$-space $X$ is \emph{flat} if the
  latching map
  \[ L_{(\bld{n_1},\bld{n_2})}X = \colim_{(\bld{m_1},\bld{m_2}) \to
    (\bld{n_1},\bld{n_2})\in\partial(\cJ\downarrow(\bld{n_1},\bld{n_2}))}X(\bld{m_1},\bld{m_2})
  \;\;\to\;\; X(\bld{n_1},\bld{n_2})\] is a cofibration of simplicial sets for
  every object $(\bld{n_1},\bld{n_2}) $ of $\cJ$. A commutative
  $\cJ$-space monoid $M$ is flat if its underlying $\cJ$-space is.
\end{definition}
This is the counterpart to the notion of flat (or $S$-cofibrant)
symmetric spectra~\cite{Schwede_SymSp}. Cofibrant objects in the
positive $\cJ$-model structure on $\cC\cS^{\cJ}$ are
flat~\cite[Proposition 4.28]{Sagave-S_diagram}. The free
$\cJ$-space $\cJ((\bld{m_1},\bld{m_2}),-)$ with
$\cJ((\bld{m_1},\bld{m_2}),-)_{h\cJ}\simeq *$ is flat.
If $X$ is a flat $\cJ$-space, then $X\boxtimes -$ preserves
$\cJ$-equivalences between all objects~\cite[Proposition
8.2]{Sagave-S_diagram}. This is useful because $X\boxtimes Y$ captures
the homotopy type of the left derived product as soon as \emph{one of}
$X$ or $Y$ is flat

As explained for the case of $\cI$-spaces in~\cite[\S
2.24]{Sagave-S_group-compl}, the homotopy colimit functor
$(-)_{h\cJ}\colon \cS^{\cJ} \to \cS$ is a monoidal (but not symmetric
monoidal) functor. This means in particular that there is a natural
monoidal structure map
\begin{equation}\label{eq:monoidal-structure-map}
 X_{h\cJ} \times Y_{h\cJ} \to (X\boxtimes Y)_{h\cJ} 
\end{equation}
for $\cJ$-spaces $X$ and $Y$. Under a flatness hypothesis, we can use this map
to analyze the homotopy type of $(X\boxtimes Y)_{h\cJ}$:
\begin{lemma}\label{lem:monoidal-structure-map-equiv}
  If one of $X$ or $Y$ is flat, then the
  map~\eqref{eq:monoidal-structure-map} is a weak equivalence.
\end{lemma}
\begin{proof}
The same proof as in the case of $\cI$-spaces~\cite[Lemma
2.25]{Sagave-S_group-compl} applies.
\end{proof}
\subsection{Group completions of commutative
  \texorpdfstring{$\cJ$}{J}-space monoids}\label{sec:CSJ-group-compl}
We need a group completion functor for commutative $\cJ$-space monoids
in order to study the graded log structures we will define using
$\cC\cS^{\cJ}$.

In the case of $\cI$-spaces, it is easy to construct group completions
because the usual bar construction for simplicial or topological
monoids lifts to commutative $\cI$-space monoids~\cite[\S
4]{Sagave-S_group-compl}.  In $\cJ$-spaces, the situation is
different: The monoidal unit $U^{\cJ}$ is concentrated in $\cJ$-space
degree $0$, i.e., $U^{\cJ}(\bld{m_1},\bld{m_2}) = \emptyset$ unless
$m_2-m_1 =0$. So $U^{\cJ}$ is not a zero object in $\cC\cS^{\cJ}$
because it is not terminal, and there is no two sided bar construction
$B^{\boxtimes}(U^{\cJ},M,U^{\cJ})$ for general $M$.

We show in~\cite{Sagave_spectra-of-units} how one can overcome this
difficulty by constructing group completions via a localization of the
positive $\cJ$-model structure on $\cC\cS^{\cJ}$. To explain some of
the details, we recall from ~\cite[\S 3]{Sagave_spectra-of-units} that
there is a functor
\begin{equation}\label{eq:functor-gamma} \gamma \colon \cC\cS^{\cJ} \to \GammaS
\end{equation}
from $\cC\cS^{\cJ}$ into Segal's category of $\Gamma$-spaces. It
satisfies $\gamma(M)(1^+) = M_{h\cJ}$ and takes values in
\emph{special} $\Gamma$-spaces. Applying it for example to the
terminal commutative $\cJ$-space monoid $*$ defines a $\Gamma$-space
$\bof{\cJ}=\gamma(*)$ that provides an infinite delooping of the
classifying space $B\cJ=(*)_{h\cJ}$. We say that a commutative
$\cJ$-space monoid $M$ is \emph{grouplike} if the commutative monoid
$\pi_0(M_{h\cJ})$ is a group. So $M$ is grouplike if and only if
$\gamma(M)$ is \emph{very special}.
\begin{theorem}\cite[Theorem 5.5]{Sagave_spectra-of-units}
  The category $\cC\cS^{\cJ}$ admits a \emph{group completion model
    structure}. The cofibrations are those of the positive $\cJ$-model
  structure. A map $M \to N$ is a weak equivalence if
  $\gamma(M)\to\gamma(N)$ is a stable equivalence of
  $\Gamma$-spaces. The fibrant objects are the positive $\cJ$-fibrant
  objects that are grouplike.
\end{theorem}
In particular, a fibrant replacement functor for this model structure
defines a functorial group completion 
$M\rightarrowtail M^{\gp}$ for commutative $\cJ$-space monoids. Weak
equivalences between grouplike objects are $\cJ$-equivalences.

We recall an example for a group completion that will become relevant
later:
\begin{example}\cite[Example
5.8]{Sagave_spectra-of-units}\label{ex:group-completion-of-free}
Let \begin{equation} \label{eq:free-com-J} M =
  \textstyle\coprod_{n\geq 0}
  (\cJ((\bld{m_1},\bld{m_2}),-))^{\boxtimes n}/\Sigma_n \end{equation}
be the free commutative $\cJ$-space monoid on a point in degree
$(\bld{m_1},\bld{m_2})$. We set $m=m_2-m_1$ and assume $m\neq 0$ and
$m_1 > 0$. Then $M$ is concentrated in $\cJ$-space degrees $mk$ for
$k\in\mN_0$, i.e., $M(\bld{n_1},\bld{n_2})=\emptyset$ unless $n_2-n_1$
is of the form $mk$. The same argument as in~\cite[Example
3.7]{Sagave-S_group-compl} shows that $M_{h\cJ} \simeq
\textstyle\coprod_{n\geq 0}B\Sigma_n$, and it follows from the
Barratt-Priddy-Quillen theorem that the group completion $(M)_{h\cJ}
\to (M^{\gp})_{h\cJ}$ is homotopic to $\textstyle\coprod_{n\geq
  0}B\Sigma_n \to QS^0$.

The augmentation $M_{h\cJ} \to (*)_{h\cJ}\iso B\cJ$ associated with
$M$ maps the generator $\id_{(\bld{m_1},\bld{m_2})}$ to the component
of $m \in \pi_0(B\cJ)\iso \mZ$. It follows that the augmentation
$ QS^0 \simeq (M^{\gp})_{h\cJ} \to B\cJ\simeq QS^0$ is multiplication
with $m$.
\end{example}

\section{Symmetric ring spectra with logarithmic
  structures}\label{sec:graded-log-ring-spectra}
In the first part of this section we introduce the basic definitions,
examples, and constructions for the theory of graded topological
logarithmic structures. In many cases these are the straightforward
modifications of the corresponding notions for topological logarithmic
structures studied by Rognes~\cite[\S 7]{Rognes_TLS}. The difference
is that we replace the commutative $\cI$-space monoids (modeling
$E_{\infty}$ spaces) used in Rognes' $\cI$-space log structures theory
by commutative $\cJ$-space monoids (modeling \emph{graded}
$E_{\infty}$ spaces).  Both Rognes' and our notions may be viewed as
homotopical generalizations of the affine version of the corresponding
concepts in logarithmic geometry as for example described by
Kato~\cite[\S 1]{Kato_logarithmic-structures} and
Ogus~\cite{Ogus_log-book}.

In the second part of this section, we give a first result in which
our graded log structures show a desirable behavior which is not
shared by $\cI$-space pre-log structures: We prove that the
localization of the free graded pre-log structure on a ring spectrum
inverts the homotopy class of its generator in the homotopy groups of
the ring spectrum.

\subsection{Pre-log structures, log structures, and logification}
Throughout this section, we let $A$ be a positive fibrant commutative
symmetric ring spectrum.  Using the adjunction
$(\mS^{\cJ}[-],\Omega^{\cJ})$ relating commutative $\cJ$-space monoids
and commutative symmetric ring spectra, we can state
\begin{definition}\label{def:graded-pre-log}
  A \emph{graded pre-log structure} $(M,\alpha)$ on $A$ is a
  commutative $\cJ$-space monoid $M$ together with a map $\alpha\colon
  M \to \Omega^{\cJ}(A)$ in $\cC\cS^{\cJ}$. If $(M,\alpha)$ is a
  graded pre-log structure on $A$, the triple $(A,M,\alpha)$ is a
  \emph{graded pre-log ring spectrum}. We write $(A,M)$ for $(A,M,\alpha)$ if
  $\alpha$ is understood from the context.

  A map $(f,f^{\flat})\colon (A,M,\alpha) \to (B,N,\beta)$ of graded
  pre-log ring spectra consists of a map $f\colon A \to B$ of
  commutative symmetric ring spectra and a map $f^{\flat}\colon M \to N$ of
  commutative $\cJ$-space monoids such that $\beta f^{\flat} =
  (\Omega^{\cJ}(f))\alpha$. We call $(f,f^{\flat})$ a \emph{weak
    equivalence} if $f$ is a stable equivalence of symmetric spectra
  and $f^{\flat}$ is a $\cJ$-equivalence.
\end{definition}
\begin{remark}
  We explained in the introduction that the notion of pre-log and log
  structures defined here differs from Rognes' notion
  in~\cite{Rognes_TLS}. We therefore use the attribute ``$\cI$-space'' for the
  various notions of pre-log and log structures considered
  in~\cite{Rognes_TLS} to distinguish them from their graded
  counterparts studied in the present paper. When there is no risk of
  confusion with Rognes' definition, we occasionally drop the
  ``graded'' in our notion to ease notation.
\end{remark}

\begin{example}\label{ex:free-pre-log}
  Let $x \colon S^{n_2} \to A_{n_1}$ be a basepoint preserving map. It
  represents a homotopy class $[x] \in \pi_{n_2-n_1}(A)$, and since
  $A$ is positive fibrant every homotopy class in $\pi_n(A)$ can be
  represented by such a map if $n_1\geq 1$ and $n=n_2-n_1$. Since
  $\Omega^{\cJ}(A)(\bld{n_1},\bld{n_2})=\Omega^{n_2}(A_{n_1})$, we may
  view $x$ as a point in $\Omega^{\cJ}(A)(\bld{n_1},\bld{n_2})$.
  
  By adjunction, $x$ induces a map
  \begin{equation}\label{eq:free-pre-log} \alpha \colon C(x) = \textstyle\coprod_{i\geq 0}
    (\cJ((\bld{n_1},\bld{n_2}),-))^{\boxtimes i}/\Sigma_i  \to
    \Omega^{\cJ}(A) \end{equation} 
  from the free commutative $\cJ$-space monoid on a point in degree
  $(\bld{n_1},\bld{n_2})$ to $\Omega^{\cJ}(A)$. This map defines the
  \emph{free graded pre-log structure} on $A$.  We write $(A,C(x))$ for
  the resulting graded pre-log ring spectrum.
\end{example}

\begin{example}
  If $M$ is a commutative $\cJ$-space monoid, the adjunction unit of
  $(\mS^{\cJ}[-],\Omega^{\cJ})$ induces the \emph{canonical graded
    pre-log structure} $M \to \Omega^{\cJ}(\mS^{\cJ}[M]^{\text{fib}})$
  on a fibrant replacement $\mS^{\cJ}[M]^{\text{fib}}$ of
  $\mS^{\cJ}[M]$.
\end{example}

\begin{example}\label{ex:direct-image-pre-log}
  Let $(B,N,\beta)$ be a graded pre-log ring spectrum and let $f
  \colon A \to B$ a map of commutative symmetric ring spectra. The
  pullback diagram
\[\xymatrix@-1pc{
f_*N \ar[rr]^-{f_*\beta} \ar[d] & & \Omega^{\cJ}(A) \ar[d] \\
N \ar[rr]^-{\beta} & & \Omega^{\cJ}(B)
}
\]
provides a graded pre-log structure $(f_*N,f_*\beta)$ on
$A$. Following the terminology of~\cite[Definition 7.26]{Rognes_TLS},
we call $(f_*N,f_*\beta)$ the \emph{graded direct image pre-log
  structure} induced by $f$ and $(N,\beta)$. It comes with a canonical
map $(A,f_*N) \to (B,N)$. Because $\cC\cS^{\cJ}$ is right proper,
assuming that $\beta$ is a positive $\cJ$-fibration or that $f$ is a
positive fibration in $\cC\Spsym$ ensures that this preserves weak
equivalences. (The attribute ``direct image'' is taken from the
algebraic notion of log rings where the variance refers to the
associated affine log schemes~\cite{Kato_logarithmic-structures}.)
\end{example}

\begin{definition}
  Let $(M,\alpha)$ be a graded pre-log structure on a positive fibrant
  commutative symmetric ring spectrum $A$ and consider the pullback
  diagram
\begin{equation}\label{eq:pullback-log-definition}
\xymatrix@-1pc{
\alpha^{-1} \GLoneJof(A) \ar[rr]^-{\widetilde{\alpha}} \ar[d]& & \GLoneJof(A) \ar[d]\\
M  \ar[rr]^-{\alpha}&  &\Omega^{\cJ}(A).
}
\end{equation}
Then $(M,\alpha)$ is a \emph{graded log structure} if
$\widetilde{\alpha}$ is a $\cJ$-equivalence. A \emph{graded log ring
  spectrum} $(A,M)$ is a graded pre-log ring spectrum with a graded
log-structure.
\end{definition}
In the definition, our assumption that $A$ is positive fibrant ensures
that $\Omega^{\cJ}(A)$ and $\GLoneJof(A)$ capture the desired homotopy
types. Since $\GLoneJof(A) \to \Omega^{\cJ}(A)$ is an inclusion of
path components, it is a positive
$\cJ$-fibration. Hence~\eqref{eq:pullback-log-definition} is homotopy
cartesian because $\cC\cS^{\cJ}$ is right proper, and condition of
$(A,M)$ being a log ring spectrum is homotopy invariant.

\begin{example}\label{ex:trivial-log}
  The inclusion $\GLoneJof(A)\to\Omega^{\cJ}(A)$ provides the
  \emph{trivial graded log structure} on~$A$.
\end{example}

\begin{example}\label{ex:direct-image-log}
  Let $(B,N)$ be a graded log ring spectrum and let $f \colon A \to B$
  be a map of positive fibrant commutative symmetric ring
  spectra. Forming the base change of $\beta\colon N \to
  \Omega^{\cJ}(B)$ along \[\GLoneJof(A) \to \GLoneJof(B) \to
  \Omega^{\cJ}(B)\quad \text{ and } \quad \GLoneJof(A) \to
  \Omega^{\cJ}(A) \to \Omega^{\cJ}(B)\] shows that that the graded
  direct image pre-log structure $f_*N$ of
  Example~\ref{ex:direct-image-pre-log} is a log structure if $\beta$
  is a positive $\cJ$-fibration or $f$ is a positive fibration in
  $\cC\Spsym$.
\end{example}

\begin{example}
  Combining the last two examples, any map $f \colon A \to B$ of
  positive fibrant objects in $\cC\Spsym$ gives rise to the graded log
  structure $f_*(\GLoneJof(B)) \to \Omega^{\cJ}(A)$ on $A$. We call
  this the \emph{graded direct image log structure} induced by
  $f$. This applies for instance to the map of complex $K$-theory
  spectra $i\colon ku \to KU$ which exhibits the connective $ku$ as
  the connective cover of the periodic $KU$.  We will study the
  properties of this log structure in
  Section~\ref{sec:direct-image-log}. As mentioned in the
  introduction, the $\cI$-space counterpart of this construction only
  provides the trivial log structure since both
  $\Omega^{\cI}(ku)\to\Omega^{\cI}(KU)$ and $\GLoneIof(ku) \to
  \GLoneIof(KU)$ are $\cI$-equivalences.

\end{example}
There is a logification process turning pre-log structures into log structures: 
\begin{construction}\label{cons:construction-logification}
If 
$(M,\alpha)$ is a graded pre-log structure on $A$, we choose a (functorial) factorization
\begin{equation}\label{eq:construction-logification}
\xymatrix@1{\alpha^{-1} \GLoneJof(A)\;\; \ar@{>->}[r] & \;G\; \ar@{->>}[r]^-{\sim} & \GLoneJof(A)}
\end{equation}
of $\widetilde{\alpha}$ and define $M^{a}$ to be the pushout in $\cC\cS^{\cJ}$ 
displayed on the left hand side of 
\[
\xymatrix@-1pc{
\alpha^{-1}\GLoneJof(A)\;\; \ar@{>->}[r] \ar[d]&   \;G\;\ar@{->>}[rr]^-{\sim} \ar[d]& &\GLoneJof(A) \ar[d] \\
M \ar[r] \ar@/_/@<-.5ex>[rrr] & M^a \ar@{-->}[rr]^{\alpha^{a}} & & \Omega^{\cJ}(A)
}
\]
Hence $M^{a}=M\boxtimes_{\alpha^{-1} \GLoneJof(A)}G$, and the
universal property of the pushout provides the map $\alpha^{a}\colon
M^a \to \Omega^{\cJ}(A)$. This construction preserves weak
equivalences: Left properness of $\cC\cS^{\cJ}$ ensures that the
pushout coincides with the homotopy pushout.  

The pre-log structure $(M^a,\alpha^{a})$ is the
\emph{associated graded log structure} of $(A,M)$, and the canonical
map $(M,\alpha) \to (M^{a},\alpha^a)$ is the \emph{logification} of
$(M,\alpha)$.
\end{construction}
This terminology is justified by the following
\begin{lemma}\label{lem:logification}
  The associated graded log structure $(M^a,\alpha)$ is a 
  log-structure. If $(A,M)$ is a graded log ring spectrum, then the
  logification is a weak equivalence.
\end{lemma}
\begin{proof}
  For brevity we write $W=\Omega^{\cJ}(A)$. Let $W^{\times} =
  \GLoneJof(A)$ be union of invertible path components, and let
  $\widehat{W}$ be the complement of $W^{\times}$, i.e., the
  $\cJ$-space given by the components of $W$ which are not invertible.
  For a map $N \to W$ in $\cC\cS^{\cJ}$, we let $N = \widetilde{N}
  \coprod \widehat{N}$ be the decomposition of the underlying
  $\cJ$-space of $N$ into the part $\widetilde{N} = N\times_W
  W^{\times}$ that maps to the units and the part $\widehat{N} =
  N\times_W \widehat{W}$ that maps to the non-units. Then
  $\alpha^{-1}\GLoneJof(A) = \widetilde{M}$, and there are
  isomorphisms
  \[ M^a = M \boxtimes_{\widetilde{M}} G \iso (\widetilde{M}
  \textstyle\coprod \widehat{M})\boxtimes_{\widetilde{M}} G \iso G
  \textstyle\coprod (\widehat{M}\boxtimes_{\widetilde{M}}G).\] Since
  $G$ maps to the units and $\widehat{M}\boxtimes_{\widetilde{M}}G$
  maps to the non-units, this shows $\widetilde{M^a} \iso G$. So
  $(M^a,\alpha^a)$ is a log structure.

The second assertion is clear because the cofibration
$\alpha^{-1}\GLoneJof(A) \rightarrowtail G$ is a $\cJ$-equivalence if $(M,A)$ is a
log ring spectrum.
\end{proof}
The proof of the previous lemma also shows
\begin{lemma}
Let $(A,M)$ be graded pre-log structure. If $M \to \Omega^{\cJ}(A)$ factors
through $\GLoneJof(A) \to \Omega^{\cJ}(A)$, then $(A,M^a)$ is weakly equivalent
to the trivial graded log structure. 
\end{lemma}

\subsection{Log structures and localization}
In this paragraph we consider a positive fibrant commutative symmetric
ring spectrum $A$ with a graded pre-log structure $(M,\alpha)$.

As discussed in Section~\ref{sec:CSJ-group-compl}, we can form the group completion
$M \rightarrowtail M^{\gp}$ of~$M$. Combining this map with $\alpha \colon M \to \Omega^{\cJ}(A)$ and
using the left adjoint $\mS^{\cJ}[-]$ of $\Omega^{\cJ}$, we obtain the
following diagram in $\cC\Spsym$:
\begin{equation}\label{eq:pushout-for-localization}
\xymatrix@-1pc{ \mS^{\cJ}[M^{\gp}] & \;\mS^{\cJ}[M] \ar@{>->}[l] \ar[r] & A }
\end{equation}
\begin{definition}
  The pushout $A[M^{-1}]=\mS^{\cJ}[M^{\gp}]\sm_{\mS^{\cJ}[M]}A$ of the
  diagram~\eqref{eq:pushout-for-localization} in the category of
  commutative symmetric ring spectra is the \emph{localization} of the
  graded pre-log ring spectrum $(A,M)$.
\end{definition}
Since the group completion is defined as a functorial fibrant
replacement, this is functorial. If $M$ is cofibrant, the fact that $M
\to M^{\gp}$ is a cofibration ensures that the localization sends weak
equivalences of pre-log ring spectra to stable equivalences. To ensure
this desirable property, we implicitly form a cofibrant replacement of
$M$ whenever we consider $A[M^{-1}]$ for a non-cofibrant $M$.

The localization of $(A,M)$ does only depend on its logification:
\begin{lemma}\label{lem:localization-and-logification} 
  If $M$ is cofibrant, then the logification $(A,M) \to (A,M^{a})$ induces
  a stable equivalence $A[M^{-1}] \to A[(M^a)^{-1}]$.
\end{lemma}
\begin{proof}
By definition, the right hand square in
\[\xymatrix@-1pc{
  \mS^{\cJ}[M] \ar[d] \ar[r] & \mS^{\cJ}[M^a] \ar[d] \ar[r]& \ar[d]A\\
  \mS^{\cJ}[M^{\gp}] \ar[r]& \mS^{\cJ}[(M^a)^{\gp}]
  \ar[r]& A[(M^a)^{-1}] }\] is a pushout. Let $Q$ be the pushout of
$M^{\gp} \ot M \to M^a$.  Then $Q$ and $(M^a)^{\gp}$ are
cofibrant since $M$ is.  Since $\mS^{\cJ}[-]$ sends $\cJ$-equivalences
between cofibrant objects to stable equivalences, the diagram shows that it
is enough to show that the canonical map $Q \to (M^a)^{\gp}$ is
a $\cJ$-equivalence. For this we consider the iterated pushout
\[\xymatrix@-1pc{
  \alpha^{-1}\GLoneJof(A) \ar[d] \ar@{>->}[r]&G \ar[d]\\
  M \ar[d] \ar@{>->}[r]&M^a \ar[d]\\
  M^{\gp}\ar@{>->}[r]& Q. }\] Both $M \to M^{\gp}$ and $M^a \to
(M^a)^{\gp}$ are acyclic cofibrations in the group completion
model structure.  By cobase change this holds for $M^a \to Q$, and two
out of three shows that $Q \to (M^a)^{\gp}$ is a weak equivalence in
the group completion model structure. We need to show that $Q$ is
grouplike in order to see that it is a $\cJ$-equivalence.  By
definition, $Q$ is grouplike if $Q_{h\cJ}$ is. The canonical map
\begin{equation}\label{eq:logification-pushout-grouplike}\pi_0((M^{\gp} \boxtimes G)_{h\cJ}) \to \pi_0((M^{\gp} \boxtimes_{\alpha^{-1}\GLoneJof(A)} G)_{h\cJ}) \iso \pi_0(Q_{h\cJ})
\end{equation}
is clearly surjective. Because $M^{\gp}$ is flat,
Lemma~\ref{lem:monoidal-structure-map-equiv} provides natural
isomorphism $\pi_0((M^{\gp})_{h\cJ}) \times \pi_0(G_{h\cJ}) \iso
\pi_0((M^{\gp} \boxtimes G)_{h\cJ})$. Since $M^{\gp}$ and $G$ are
grouplike, the domain of the
surjection~\eqref{eq:logification-pushout-grouplike} is a group, and
so is $\pi_0(Q_{h\cJ})$.
\end{proof}

We recall from~\cite{Schwede_SymSp} that the process of inverting an
element in a graded ring can be lifted from homotopy groups to symmetric ring
spectra:
\begin{proposition}\label{prop:inverting-element}
Let $x \colon S^{n_2} \to A_{n_1}$ represent a homotopy class $[x] \in \pi_{n_2-n_1}(A)$.
Then $x$ gives rise to a positive fibrant commutative symmetric ring spectrum $A[1/x]$ and
a positive cofibration $i\colon A \to A[1/x]$ in $\cC\Spsym$ 
which induces an isomorphism
$\pi_*(A)[1/[x]] \to \pi_*(A[1/x])$. This construction is natural in $A$.  
\end{proposition}
\begin{proof}
  An explicit representative for $A \to A[1/x]$ is constructed
  in~\cite[I.Corollary 4.69]{Schwede_SymSp}. This map has the desired
  properties apart from being a cofibration with fibrant
  codomain. Forming a fibrant replacement of the codomain and
  factoring the map from $A$ into the fibrant replacement by a
  cofibration followed by an acyclic fibration provides the desired
  cofibration $i \colon A \to A[1/x]$ with fibrant codomain.
\end{proof}
The following universal property of this construction is also discussed
in~\cite{Schwede_SymSp}:
\begin{corollary}\label{cor:universal-property-inverting-element}
Let $A \to B$ be a map into a positive fibrant commutative symmetric ring
spectrum $B$ which sends $[x]$ to a unit in $\pi_*(B)$. Then $A \to B$ extends
over $A \to A[1/x]$. 
\end{corollary}
One can use the statement of
Proposition~\ref{prop:localization-of-free} and the argument given in
the last part of its proof to see that two possible extensions $A[1/x]
\to B$ in the corollary are homotopic relative $A$.
\begin{proof}[Proof of Corollary~\ref{cor:universal-property-inverting-element}]
Inverting $x$ and its image in $B$ we obtain a commutative square
\[\xymatrix@-1pc{A \ar[r] \ar[d] & B \ar[d]\\
  A[1/x] \ar[r] & B[1/x]}\] in which the map $B \to B[1/x]$ is a
stable equivalence. So we get a map $A[1/x] \to B$ in the homotopy
category of $A$-algebras. Since $A[1/x]$ is cofibrant as an
$A$-algebra and $B$ is fibrant, we can realize it as a map of
$A$-algebras.
\end{proof}

In many interesting cases, we can now identify the localizations of
the free graded pre-log structures introduced in
Example~\ref{ex:free-pre-log}:
\begin{proposition}\label{prop:localization-of-free}
  Let $A$ be a positive fibrant commutative symmetric ring spectrum,
  and let $x \colon S^{n_2} \to A_{n_1}$ represent a homotopy class
  $[x] \in \pi_{*}(A)$ of degree $n=n_2-n_1$. Then the localization
  $A[(C(x))^{-1}]$ of the free graded pre-log structure $C(x)$
  associated with $x$ is stably equivalent to the commutative
  symmetric ring spectrum $A[1/x]$.
\end{proposition}

\begin{remark}
  The proposition exhibits one of the key features of graded log
  structures: The map $x$ also generates a free $\cI$-space pre-log
  structure on $S^n$~\cite[Example 7.18]{Rognes_TLS}. However, the
  localization of this $\cI$-space pre-log structure is a homotopy
  pushout of connective ring spectra, and will not give $A[1/x]$ if $n
  > 0$. In contrast, our graded setup shows a behavior that one would
  naively expect from algebra: If $A$ is a commutative ring, then an
  element $a \in A$ gives rise to a map $\mN_0 \to (A,\cdot)$ from the
  free commutative monoid on one generator into the multiplicative
  monoid of $A$. This determines a ring map $\mZ[x]\iso\mZ[\mN_0] \to
  A$, and the pushout of $\mZ[(\mZ,+)] = \mZ[x^{\pm 1}] \ot \mZ[x]\to
  A$ is the ring $A[1/a]$.
\end{remark}
\begin{proof}[Proof of Proposition~\ref{prop:localization-of-free}]
  Since $[x]$ is a unit in $\pi_*(A[1/x])$, we can extend the
  composite $C(x) \to \Omega^{\cJ}(A) \to \Omega^{\cJ}(A[1/x])$ over
  $C(x)  \rightarrowtail  C(x)^{\gp}$. The adjunction
  $(\mS^{\cJ}[-],\Omega^{\cJ})$ turns this into maps in $\cC\Spsym$,
  and the universal property of the pushout provides a map
  $A[C(x)^{-1}] \to A[1/x]$. Factoring this map as
  \[\xymatrix{A[C(x)^{-1}] \;\;\ar@{>->}^-{j}_-{\simeq}[r] &\; P \ar@{->>}[r] &
    A[1/x]}\] gives a fibrant replacement $P$ of $A[C(x)^{-1}]$, and it is
  enough to show that the map of $A$-algebras $P \to A[1/x]$ is a
  stable equivalence.

  We first show that the map induced by the composite $A \to
  A[C(x)^{-1}] \to P$ sends the homotopy class of $x$ to a unit in
  $\pi_*(P)$: The image of $[x]$ is invertible in $\pi_*(P)$ if and
  only if the image of $x$ in $\pi_0((\Omega^{\cJ}P)_{h\cJ})$
  represents a unit. The latter condition is satisfied since by
  construction, the image of $x$ in $\pi_0((\Omega^{\cJ}P)_{h\cJ})$
  lies in the image of the homomorphism $\pi_0((C(x)^{\gp})_{h\cJ})
  \to \pi_0((\Omega^{\cJ}P)_{h\cJ})$ whose domain is a group. Hence
  Corollary~\ref{cor:universal-property-inverting-element} provides a
  map $A[1/x] \to P$ of $A$-algebras.

  The composite $A[1/x] \to P \to A[1/x]$ is a stable equivalence: The
  induced map of homotopy groups $\pi_*(A)[1/[x]]\iso\pi_*(A[1/x]) \to
  \pi_*(A[1/x])\iso\pi_*(A)[1/[x]]$ is a map under $\pi_*(A)$ and
  hence an isomorphism. So it remains to show that the composite $P
  \to A[1/x]\to P$ is a stable equivalence. Let us write $k$ for this
  map. We use the acyclic cofibration $j\colon A[C(x)^{-1}] \to P$
  defined above and claim that $j$ and $kj$ are homotopic as maps of
  $A$-algebras. By adjunction, these maps correspond to two possible
  extensions of $C(x) \to \GLoneJof(P)$ along $C(x) \to
  C(x)^{\gp}$. Applying the homotopy uniqueness of lifts in model
  categories~\cite[Proposition 7.6.13]{Hirschhorn_model} to this
  extension problem in the group completion model structure shows
  that the extensions are homotopic, and it follows that $j$ and $kj$
  are homotopic. So $k$ is a stable equivalence because $j$ is.
\end{proof}

\section{Direct image log structures generated by homotopy
  classes}\label{sec:direct-image-log}
In this section we compare various ways of how classes in the homotopy
groups of connective ring spectra give rise to graded pre-log and log
structures and apply this to the case of $K$-theory spectra.

\subsection{Log structures on connective ring spectra}
As before we let $A$ be a positive fibrant commutative symmetric ring
spectrum. In addition, we now assume that $A$ is connective and fix a
map $x \colon S^{n_2} \to A_{n_1}$ with $n_1\geq 1$ that represents a
non-trivial homotopy class $[x] \in \pi_*(A)$ of even positive degree
$n=n_2-n_1$ such that the localization map $i\colon A \to A[1/x]$ of
Proposition~\ref{prop:inverting-element} exhibits $A$ as the
connective cover of~$A[1/x]$. This means that $i\colon A \to A[1/x]$
induces an isomorphism on stable homotopy groups of non-negative
degrees. The connective complex $K$-theory spectrum and the Bott class
are the motivating example for this setup.

Since $A$ is positive fibrant, the existence of such a representing
map $x$ for the relevant homotopy class is always ensured. Our
assumptions in particular imply that $A$ represents a non-trivial
homotopy type and that $[x]$ is not already a unit in $\pi_*(A)$. So
$x$ has a chance to generate a non-trivial log structure.

We have already introduced the free graded pre-log structure $C(x)$ on
$A$ (Example~\ref{ex:free-pre-log}) and the graded direct image log
structure $i_*\!\GLoneJof{A[1/x]}$ induced by $i\colon A \to A[1/x]$
(Example~\ref{ex:direct-image-log}). To clarify their properties and
relationship, we consider another less obvious but highly useful
graded pre-log structure arising in this setup:

\begin{construction}\label{cons:graded-direct-pre-log-from-x}
  The following diagram summarizes the various steps to be described
  next. They will lead to a graded pre-log structure $D(x)$ on $A$.\\
  \begin{equation}\label{eq:graded-direct-pre-log-from-x}
    \xymatrix{
      C(x) \ar@{>->}[r] \ar[d]\ar@/^.7pc/@<1ex>[rrrr]&  D(x)\ar@{->>}[r]^{\sim} &  D'(x)\ar[rr]\ar[dll]& & \Omega^{\cJ}(A)\ar[d]\\ 
      C(x)^{\gp}\ar[rrr]&&& \GLoneJof(A[1/x])\ar[r]& \Omega^{\cJ}(A[1/x])
    }
  \end{equation}  
  To obtain $D(x)$, we first observe that the composite
  \[ C(x) \to \Omega^{\cJ}(A) \to \Omega^{\cJ}(A[1/x]) \] factors
  through $\GLoneJof(A[1/x]) \to \Omega^{\cJ}(A[1/x])$ because $[x]$
  becomes invertible in $\pi_*(A[1/x])$. We then factor the resulting
  map $C(x) \to \GLoneJof(A[1/x])$ in the group completion model
  structure as an acyclic cofibration $C(x) \to C(x)^{\gp}$ followed
  by a fibration $C(x)^{\gp} \to \GLoneJof(A[1/x])$. Since
  $\GLoneJof(A[1/x])$ is positive $\cJ$-fibrant and grouplike, so is
  $C(x)^{\gp}$, and $C(x) \to C(x)^{\gp}$ is a possible choice for the
  group completion of~$C(x)$. 

  Next we form the pullback $D'(x)$ of $C(x)^{\gp} \twoheadrightarrow
  \Omega^{\cJ}(A[1/x]) \ot \Omega^{\cJ}(A)$ to obtain a pre-log
  structure $D'(x)$ on $\Omega^{\cJ}(A)$. A cofibrant replacement of
  $D'(x)$ relative to $C(x)$ defines the desired pre-log structure
  $D(x)$ on $\Omega^{\cJ}(A)$.

  We call $D(x)$ the \emph{graded direct image pre-log structure}
  associated with $x$. There will be no confusion with the more
  general direct image pre-log structure discussed in
  Example~\ref{ex:direct-image-pre-log}, although $D'(x)$ may of
  course be viewed as the direct image pre-log structure
  induced by $i\colon A \to A[1/x]$ and $C(x)^{\gp} \to
  \Omega^{\cJ}(A[1/x])$.
\end{construction}
\begin{remark}
  In the construction, one could also extend $C(x) \to
  \GLoneJof(A[1/x])$ over a given group completion $C(x)
  \rightarrowtail C(x)^{\gp}$ of $C(x)$ by means of the lifting axiom
  in the group completion model structure. Defining $C(x)^{\gp}$ as we
  did has the advantage of being functorial and making both $D'(x)$
  and $D(x)$ homotopy invariant.
\end{remark}

By the universal property of the coproduct $\boxtimes$ in
$\cC\cS^{\cJ}$, the pre-log structure $C(x) \to
\Omega^{\cJ}(A)$ and the trivial log structure $\GLoneJof(A)
\to \Omega^{\cJ}(A)$ give rise to a pre-log ring spectrum
$(A,C(x)\boxtimes\GLoneJof(A))$. The diagram of commutative
$\cJ$-space monoids~\eqref{eq:graded-direct-pre-log-from-x} now
induces the following commutative diagram of graded pre-log ring
spectra:
\begin{equation}\label{eq:all-pre-log-on-A}
\xymatrix@-1pc{
(A,C(x)) \ar[r] \ar[d] & (A,D(x)) \ar[d] \ar[dr] \\
(A,C(x)\boxtimes\GLoneJof(A)) \ar[r] & (A,i_*\!\GLoneJof(A[1/x]) \ar[r] & (A[1/x], \GLoneJof(A[1/x])
}
\end{equation}
The next theorem summarizes the properties of these pre-log ring
spectra:
\begin{theorem}\label{thm:all-pre-log-on-A}
  Let $A$ be a connective positive fibrant commutative symmetric ring
  spectrum and let $x \colon S^{n_2}\! \to\! A_{n_1}$ represent a
  non-trivial homotopy class \mbox{$[x] \in \pi_*(A)$} of even
  positive degree $n=n_2-n_1$ such that the localization map $i\colon
  A \to A[1/x]$ exhibits $A$ as the connective cover of~$A[1/x]$.

  Then the graded pre-log structures in the bottom line
  of~\eqref{eq:all-pre-log-on-A} are log structures, the two vertical
  maps induce weak equivalences after logification, and the
  localization of all four graded pre-log and log structures on $A$ is
  stably equivalent to $A[1/x]$.
\end{theorem}
In other words, $x$ induces two different graded log structures
$C(x)\boxtimes\GLoneJof(A)$ and $i_*\!\GLoneJof(A[1/x])$ on $A$ which
arise as the logification of ``smaller'' graded pre-log structures
$C(x)$ and $D(x)$. Both these log structures approximate $A[1/x]$ in
that their localizations are equivalent to $A[1/x]$. However, we will
see that $i_*\!\GLoneJof(A[1/x])$ is the more useful one: It is more
canonical in that it does not depend on the choice of the
representative $x$, and we employ it for our results about formally
graded log \'{e}tale extensions in Section~\ref{sec:log-etale-maps}.

The reason for introducing $D(x)$ is that it is a convenient
presentation of the log-structure $i_*\!\GLoneJof(A[1/x])$ that is
easier to work with. This will become clear in the proof
Theorem~\ref{thm:ell-to-k-formally-etale} and in the proof of
Theorem~\ref{thm:all-pre-log-on-A} at the end of this paragraph.

\begin{remark}
  The analogy with the situation for discrete valuation rings allows
  for the following geometric interpretation of the previous theorem:
  The localization $A \to A[1/x]$ corresponds to an open immersion
  $j\colon\mathrm{spec}(A[1/x]) \to \mathrm{spec}(A)$. The maps
  in~\eqref{eq:all-pre-log-on-A} give rise to various factorizations
  \[ \mathrm{spec}(A[1/x]) \to \mathrm{spec}(A,M) \to
  \mathrm{spec}(A)\] of $j$ in the category of log schemes that
  provide relative compactifications of $j$.

  In the evolving subject of derived algebraic geometry, one often
  considers sheaves of \emph{connective} $E_{\infty}$ spectra. The
  above suggests that graded log structures might be useful for
  treating sheaves of \emph{periodic} $E_{\infty}$ ring spectra.
\end{remark}
In order to prove the theorem, we begin with giving a more explicit
description of the homotopy type of the commutative $\cJ$-space monoid
$D(x)$. 

\begin{lemma}\label{lem:Cx-to-Dx_on-hJ}
  The space $ D(x)_{h\cJ}$ is weakly equivalent to $(QS^0)_{\geq 0}$,
  and the map $C(x)_{h\cJ} \to D(x)_{h\cJ}$ is homotopic to the map
  $\coprod_{k\geq 0}B\Sigma_k \to (QS^0)_{\geq 0}$ obtained by
  restricting to the non-negative path components in the group
  completion $QS^0$ of $C(x)_{h\cJ}$. 
\end{lemma}
\begin{proof}
  It is enough to consider $D'(x)$. The space $D'(x)(\bld{m_1},\bld{m_2})$
  is the pullback of
  \[(C(x)^{\gp})(\bld{m_1},\bld{m_2}) \to
  \Omega^{\cJ}(A[1/x])(\bld{m_1},\bld{m_2}) \ot
  \Omega^{\cJ}(A)(\bld{m_1},\bld{m_2}).\]
  If $m_2-m_1$ is negative, the pullback is empty because
  the image of a point in $D'(x)(\bld{m_1},\bld{m_2})$ in
  $\Omega^{\cJ}(A)(\bld{m_1},\bld{m_2})$ would represent a power of an
  inverse of $[x]$ in $\pi_*(A)$. But our assumptions imply that $[x]$
  is not a unit in $\pi_*(A)$.

  Since $\pi_*(A) \to \pi_*(A[1/x])$ is an isomorphism in non-negative
  degrees, the map $\Omega^{\cJ}(A) \to \Omega^{\cJ}(A[1/x])$ is a
  weak equivalence when evaluated on $(\bld{m_1},\bld{m_2})$ with
  $m_1\geq 1$ and $m_2-m_1\geq 0$. The same holds for the base change
  $D'(x) \to C(x)^{\gp}$ because $C(x)^{\gp} \to \Omega^{\cJ}(A[1/x])$
  is a positive fibration by construction. Since the full subcategory
  of $\cJ$ on the objects $(\bld{m_1},\bld{m_2})$ with $m_1\geq 1$ is
  a homotopy cofinal subcategory of $\cJ$~\cite[Corollary
  5.9]{Sagave-S_diagram}, the description of the group completion
  $C(x)^{\gp}$ of $C(x)$ in Example~\ref{ex:group-completion-of-free}
  proves the claim.
\end{proof}

The following lemma is the key step towards Theorem~\ref{thm:all-pre-log-on-A}:
\begin{lemma}\label{lem:logification-of-Dx}
  The map $(A,D(x)) \to (A,i_*\!\GLoneJof(A[1/x]))$
  in~\eqref{eq:all-pre-log-on-A} induces a weak equivalence
  $(A,D(x)^a) \to (A,i_*\!\GLoneJof(A[1/x]))$.
\end{lemma}
\begin{proof}
  In is enough to show the statement for $(A,D'(x)) \to
  (A,i_*\!\GLoneJof(A[1/x]))$. 
 
  We write $\alpha \colon D'(x) \to \Omega^{\cJ}(A)$ for the structure
  map of $(A,D'(x))$. In $\cJ$-space degree $0$, its restriction
  $D'(x)_0 \to \Omega^{\cJ}(A)$ factors through $\GLoneJof(A) \to
  \Omega^{\cJ}(A)$ since the zero component of $D'(x)_{h\cJ}$ maps
  into the component of the unit of
  $(\Omega^{\cJ}(A))_{h\cJ}$. Because $A$ is connective, $
  \GLoneJof(A)$ is concentrated in $\cJ$-space degree $0$. It follows
  that the pullback $\alpha^{-1}(\GLoneJof(A))$ is isomorphic to
  $D'(x)_0$.

  Let $G$ be the replacement of $\GLoneJof(A)$ used in the
  logification of
  Construction~\ref{cons:construction-logification}. We have to show
  that the following map is a $\cJ$-equivalence:
  \[ D'(x)\boxtimes_{D'(x)_0} G \to i_*\!\GLoneJof(A[1/x]) \]
  Composing the canonical maps $\cJ((\bld{n_1},\bld{n_2}),-)^{\boxtimes k} \to C(x)$
  with $C(x) \to D'(x)$ and inducing up to a map of $D'(x)_0$-modules
  provides a map of $D'(x)_0$-modules
  \begin{equation}\label{eq:D-of-x-sum-of-shifted-copies}
    \textstyle\coprod_{k\geq 0} D'(x)_0 \boxtimes
    (\cJ((\bld{n_1},\bld{n_2}),-))^{\boxtimes k} \to D'(x).
  \end{equation} 
  Lemma~\ref{lem:Cx-to-Dx_on-hJ} and
  Lemma~\ref{lem:monoidal-structure-map-equiv} show that this map is a
  $\cJ$-equivalence. If we view \mbox{$G\boxtimes_{D'(x)_0}-$} as a
  functor from $D'(x)_0$-modules to $G$-modules, we can apply it to
  ~\eqref{eq:D-of-x-sum-of-shifted-copies}. This reduces the claim to
  showing that the composite map
  \begin{equation}\label{eq:GLoneJofA-x-inverted-decomp}
    \textstyle\coprod_{k\geq 0} G \boxtimes
    (\cJ((\bld{n_1},\bld{n_2}),-))^{\boxtimes k} \to
    i_*\!\GLoneJof(A[1/x])
  \end{equation} 
  is a $\cJ$-equivalence. Since $A$ is connective, $G$ is concentrated
  in $\cJ$-space degree~$0$, and the map $G \to
  i_*\!\GLoneJof(A[1/x])_0 \simeq \GLoneJof(A[1/x])_0$ into the
  $\cJ$-space degree~$0$ part of $i_*\!\GLoneJof(A[1/x])$ is a
  $\cJ$-equivalence.  The $i_*\!\GLoneJof(A[1/x])$ is the part of
  $\GLoneJof(A[1/x])$ sitting in non-negative $\cJ$-space
  degrees. Since multiplication with the element in
  $\GLoneJof(A[1/x])(\bld{m_1},\bld{m_2})$ represented by $x$ induces
  an equivalence between the different $\cJ$-space degree parts of
  $i_*\!\GLoneJof(A[1/x])$, it follows from
  Lemma~\ref{lem:monoidal-structure-map-equiv} that the
  map~\eqref{eq:GLoneJofA-x-inverted-decomp} is a $\cJ$-equivalence.
\end{proof}

\begin{remark}
  The proof of the previous lemma suggests that we may think of $D(x)$
  as a kind of polynomial algebra on $x$, although it is certainly not
  free. The key feature of $D(x)$ is that the group completion in its
  construction ensures that the components corresponding to the powers
  of $x$ are equivalent. Obviously, this property is not shared by the
  free object $C(x)$. This distinction between the homotopical
  counterparts ``free'' and ``polynomial'' algebras does not occur in
  algebra and is one reasons for why not all pre-log structures
  in~\eqref{eq:all-pre-log-on-A} have algebraic precursors.
\end{remark}

\begin{lemma}\label{lem:logification-of-free}
The map $(A,C(x)) \to (A,C(x)\boxtimes \GLoneJof(A))$ induces a weak equivalence
$(A,C(x)^a) \to (A,C(x)\boxtimes \GLoneJof(A))$.
\end{lemma}
\begin{proof}
  Let $\alpha \colon C(x) \to \Omega^{\cJ}(A)$ the structure map of
  $(A,C(x))$.  A similar argument as in
  Lemma~\ref{lem:logification-of-Dx} shows $\alpha^{-1}(\GLoneJof(A))
  \iso C(x)_0$. Since $C(x)_0$ is the monoidal unit, $C(x)^a =
  C(x)\boxtimes_{C(x)_0} G \iso C(x) \boxtimes G$ is $\cJ$-equivalent
  to $C(x)\boxtimes \GLoneJof(A)$.
\end{proof}
\begin{lemma}\label{lem:localization-of-Dx}
  The localization $A[D(x)^{-1}]$ of $(A,D(x))$ is stably equivalent
  to the commutative symmetric ring spectrum $A[1/x]$.
\end{lemma}
\begin{proof}
  Since the localization of $(A,C(x))$ is stably equivalent to
  $A[1/x]$ by Proposition~\ref{prop:localization-of-free}, it is
  enough to show that $(A,C(x)) \to (A,D(x))$ induces a weak
  equivalence on localizations. Arguing as in the proof of
  Lemma~\ref{lem:localization-and-logification}, this reduces to
  showing that the pushout $Q$ of $C(x)^{\gp} \ot C(x) \to D(x)$
  is grouplike. We know that there is a surjection
  $\pi_0(C(x)^{\gp}_{h\cJ}) \times \pi_0(D(x)_{h\cJ}) \to
  \pi_0(Q_{h\cJ})$. The latter map induces a surjection
  $\pi_0(C(x)^{\gp}_{h\cJ}) \tensor_{\pi_0(C(x)_{h\cJ})}
  \pi_0(D(x)_{h\cJ}) \to \pi_0(Q_{h\cJ})$, whose domain is a group
  because $\pi_0((C(x) \to D(x))_{h\cJ})$ is an isomorphism by
  Lemma~\ref{lem:Cx-to-Dx_on-hJ}.
\end{proof}
\begin{proof}[Proof of Theorem~\ref{thm:all-pre-log-on-A}]
  Proposition~\ref{prop:localization-of-free} and
  Lemma~\ref{lem:localization-of-Dx} show that the localizations of
  $(A,C(x))$ and $(A,D(x))$ are stably equivalent to $A[1/x]$.
  Lemma~\ref{lem:logification-of-Dx} and
  Lemma~\ref{lem:logification-of-free} show that the vertical maps
  in~\eqref{eq:all-pre-log-on-A} are equivalences after
  logification. Since the two logifications are log-ring spectra
  $(A,M)$ with cofibrant $M$ by construction,
  Lemma~\ref{lem:localization-and-logification} implies that the
  localizations of $(A,C(x)\boxtimes \GLoneJof(A))$ and
  $(A,i_*\!\GLoneJof(A[1/x]))$ are equivalent to $A[1/x]$.
\end{proof}
\begin{remark}\label{rem:degree-zero-case}
  One may wonder if a homotopy class $[x] \in \pi_0(A)$ of degree $0$
  gives rise to graded pre-log structures with the same properties as
  in Theorem~\ref{thm:all-pre-log-on-A}. For such $[x]$, we can still
  pick a representing map $x$ and form $D(x)$ as in
  Construction~\ref{cons:graded-direct-pre-log-from-x}. However, the
  proof of Lemma~\ref{lem:logification-of-Dx} does not apply to this
  $D(x)$ because its components are not equivalent. Nevertheless, a
  similar argument as in Lemma~\ref{lem:logification-of-Dx} shows that
  the localization of $(A,D(x))$ is stably equivalent to $A[1/x]$.
\end{remark}

\subsection{Log structures on \texorpdfstring{$K$}{K}-theory spectra}
The connective complex $K$-theory spectrum $ku$ and its $p$-local or
$p$-complete counterparts $ku_{(p)}$ and $ku_{p}$ at a prime $p$ are
$E_{\infty}$ ring spectra (see e.g.~\cite{EKMM} or~\cite{Schwede_SymSp}). They
can hence be represented by positive fibrant commutative symmetric
ring spectra. The Bott class $u \in \pi_*(ku)\iso \mZ[u]$
(respectively $u\in\pi_*(ku_{(p)})\iso \mZ_{(p)}[u]$ or
$u\in\pi_*(ku_p)\iso \mZ_p[u]$) is a homotopy class of degree $2$, and
the corresponding periodic theories $KU$, $KU_{(p)}$, and $KU_p$ can
be obtained by inverting the Bott
class. Theorem~\ref{thm:all-pre-log-on-A} implies that the periodic
spectra give rise to direct image log structures on the connective
spectra whose localization is the periodic theory.

If $p$ is odd, the same applies to the $p$-local Adams summand $\ell$
of $ku_{(p)}$ and the $p$-complete Adams summand $\ell_p$ of $ku_{p}$.
Baker and Richter have shown that $\ell$ admits a unique $E_{\infty}$
structure~\cite[Corollary 1.4]{Baker-R_uniqueness} and that the
resulting $E_{\infty}$ structure on $\ell_p$ coincides with the one
considered for example in~\cite[\S 2.1]{Ausoni_THH_ku}.  Therefore,
these spectra may be represented by positive fibrant commutative
symmetric ring spectra. This time, $v_1\in \pi_*(\ell)\iso
\mZ_{(p)}[v_1]$ (resp. $v_1\in \pi_*(\ell_p)\iso \mZ_{p}[v_1]$) is a
homotopy class of degree $2p-2$. Inverting $v_1$ gives the periodic
theories $L$ and $L_p$, and we may apply
Theorem~\ref{thm:all-pre-log-on-A} as above.

By~\cite[\S 2.1]{Ausoni_THH_ku} and~\cite[Remark 9.4]{Baker-R_numerical},
the inclusions $\ell_p \to ku_{p}$ and $\ell \to ku_{(p)}$ can be
represented by maps of $E_{\infty}$ ring spectra and therefore by maps
of commutative symmetric ring spectra. Moreover, passing to fibrant
replacements we may assume that these maps are positive fibrations of
positive fibrant commutative symmetric ring spectra $\iota_{(p)}
\colon \ell \to ku_{(p)}$ and $\iota_{p} \colon \ell_p \to
ku_{p}$. Under these assumptions, the fact that the induced map of
homotopy groups send $v_1$ to $u^{p-1}$ implies
\begin{proposition}\label{prop:log-maps-ell-ku}
The homotopy classes $u\in\pi_2(ku_p)$ and
  $v_1\in \pi_{2p-2}(\ell_{p})$ admit representatives $u\colon S^3\! \to \!(ku_{p})_1$ and
  $v_1\colon\! S^{3(p-1)}\to\! (\ell_{p})_{p-1}$ such that $\iota_{p}$ induces the
  following commutative diagram of pre-log ring spectra:
\begin{equation}\label{eq:log-maps-ell-ku}
\xymatrix@-1.2pc{
(\ell_{p}, C(v_1)) \ar[r]\ar[d]& (\ell_{p}, D(v_1)) \ar[r]\ar[d]& (\ell_{p}, i_*\!\GLoneJof(L_{p})) \ar[r]\ar[d] & (L_{p}, \GLoneJof(L_{p})) \ar[d]\\ 
(ku_{p}, C(u)) \ar[r]& (ku_{p}, D(u)) \ar[r]& (ku_{p}, i_*\!\GLoneJof(KU_{p})) \ar[r] & (KU_{p}, \GLoneJof(KU_{p})) 
}
\end{equation}
The same holds in the $p$-local case.
\end{proposition}
To ease notation, we have started to use the same symbols for the homotopy classes and
their representatives.
\begin{proof}
  Two arbitrary representatives $v_1'\colon S^{3(p-1)}\to
  (\ell_{p})_{p-1}$ and $u\colon S^{3} \to (ku_{p})_1$ of $v_1$ and
  $u$ may be viewed as points in
  $\Omega^{\cJ}(ku_{p})(\bld{1},\bld{3})$ and
  $\Omega^{\cJ}(\ell_{p})(\bld{p-1},\bld{3(p-1)})$. They have the
  property that the image of $v_1'$ under $(\iota_{p}\colon \ell_{p}
  \to ku_{p})_*$ lies in the same component of
  $\Omega^{\cJ}(ku_{p})(\bld{p-1},\bld{3(p-1)})$ as $u^{p-1}$. Since
  we assumed $\iota_p\colon \ell_p \to ku_p$ to be a positive
  fibration, we can use the path lifting property of the Kan
  fibration \[\Omega^{\cJ}(\ell_{p})(\bld{p-1},\bld{3(p-1)}) \to
  \Omega^{\cJ}(ku_{p})(\bld{p-1},\bld{3(p-1)})\] to get a
  representative $v_1\colon S^{3(p-1)}\to (\ell_{p})_{p-1}$ hitting
  $u^{p-1}$. With \mbox{$KU_p=ku_p[1/u]$} and $L_p=\ell_p[1/v_1]$ as models
  for the periodic spectra,
  Corollary~\ref{cor:universal-property-inverting-element} shows that
  $\ell_p \to ku_p$ extends to a commutative square
  \[\xymatrix@-1pc{
    \ell_p \ar[r] \ar[d] & L_p \ar[d] \\
    ku_p \ar[r] & KU_p.  } \] So the right hand square
  in~\eqref{eq:log-maps-ell-ku} commutes. The relation
  $(\iota_p)_*(v_1) = u^{p-1}$ provides the commutativity of the outer
  square in the diagram
  \[\xymatrix@-1pc{
    C(v_1) \ar@{>->}[r]^-{\sim}\ar[d] & C(v_1)^{\gp}
    \ar@{->>}[r]\ar@{-->}[d]
    & \GLoneJof(L_{p}) \ar[r]\ar[d] & \Omega^{\cJ}(L_{p}) \ar[d]\\
    C(u) \ar@{>->}[r]^-{\sim} & C(u)^{\gp} \ar@{->>}[r] &
    \GLoneJof(KU_{p}) \ar[r] & \Omega^{\cJ}(KU_{p}).  }\] in which the
  left hand vertical map sends the generator in $C(v_1)$ to the
  $(p-1)$-fold power of the generator in $C(u)$. The indicated acyclic
  cofibrations and fibrations in the diagram refer to the group
  completion model structure, and the lifting axiom in this model
  structure provides the dotted map. Passing to the pullbacks defining
  the direct image pre-log and log structures gives the commutativity
  of the diagram in the statement of the lemma. The $p$-local case
  works analogously.
\end{proof}
\begin{remark}
  The previous proposition exhibits another advantage of graded log
  structures: As explained in~\cite[Remark 7.19]{Rognes_TLS}, the free
  $\cI$-space pre-log structures on $\ell$ and $ku_{(p)}$ do not allow
  to extend $\iota_{(p)}$ to a map of pre-log ring
  spectra. In~\cite{Rognes_TLS}, Rognes develops a theory of
  \emph{based} $\cI$-space log structures in order to make this
  possible. Because of the proposition, we do not need to address a
  based version of graded log structures for this purpose.
\end{remark}

Using the adjunction $(\mS^{\cJ}[-], \Omega^{\cJ})$, the map
$(\ell_{p}, D(v_1)) \to (ku_{p}, D(u))$ of
Proposition~\ref{prop:log-maps-ell-ku} induces a commutative square in
$\cC\Spsym$:
\begin{equation}\label{eq:ku-p-as-hty-pushout}
\xymatrix@-1pc{\mS^{\cJ}[D(v_1)] \ar[r]\ar[d] & \ell_{p} \ar[d] \\
\mS^{\cJ}[D(u)]\ar[r]&ku_{p}
}
\end{equation}
The following proposition will become crucial in
Section~\ref{sec:log-etale-maps}:
\begin{proposition}\label{prop:ku-p-as-hty-pushout}
  The square~\eqref{eq:ku-p-as-hty-pushout} is a homotopy cocartesian
  square of commutative symmetric ring spectra, and the same holds in
  the $p$-local case.
\end{proposition}
As explained in~\cite[Example 12.16 and Example 12.17]{Rognes_TLS},
the counterpart of this statement for the based  $\cI$-space pre-log
structures on these spectra does not hold. It will also become clear
from the proof that this does not hold if we replace $D(v_1)$ and
$D(u)$ by the free graded log structures $C(v_1)$ and $C(u)$ on these
spectra.
\begin{proof}
  Without loss of generality, we assume that $D(v_1) \to D(u)$ is a
  cofibration of commutative $\cJ$-space monoids. In the sequel, we
  view $D(u)$ as a $D(v_1)$-module via $D(v_1) \to D(u)$ and consider
  the $D(v_1)$-module
  \[ E = \textstyle\coprod_{0\leq i\leq p-2}
  {\cJ}((\bld{1},\bld{3}),-)^{\boxtimes i} \boxtimes D(v_1).
  \]
  Composing the canonical maps $
  {\cJ}((\bld{1},\bld{3}),-)^{\boxtimes i} \to C(u)$ with $C(u) \to
  D(u)$, we obtain an induced map of $D(v_1)$-modules $E \to
  D(u)$. Since $D(u)$ is positive fibrant, the choice of a
  factorization $\xymatrix@1{E\, \ar@{>->}[r]^{\sim}& \,E'\,\ar@{->>}[r]&  \,D(u)}$ into an acyclic cofibration
  followed by a fibration in the positive $\cJ$-model structure on
  $D(v_1)$-modules provides a fibrant replacement $E'$ of
  $E$. Applying the functor $\mS^{\cJ}[-]\sm_{\mS^{\cJ}[D(v_1)]} \ell_p$
  from $D(v_1)$-modules to $\ell_p$-modules, we obtain a sequence of
  maps
  \begin{equation}\label{eq:ku-decomposition}
    \mS^{\cJ}[E]\sm_{\mS^{\cJ}[D(v_1)]} \ell_p \to \mS^{\cJ}[E']\sm_{\mS^{\cJ}[D(v_1)]} \ell_p 
    \to \mS^{\cJ}[D(u)]\sm_{\mS^{\cJ}[D(v_1)]} \ell_p
    \to ku_{p}\end{equation}
  Our cofibrancy assumptions imply that the smash products coincide
  with the derived smash products.  The claim of the proposition is
  that the last map in~\eqref{eq:ku-decomposition} is a stable equivalence, and it is enough to show
  that the first two maps and the composition of all three maps
  in~\eqref{eq:ku-decomposition} are stable equivalences.  The first
  map is a stable equivalence because $E \to E'$ is an acyclic
  cofibration, and these are preserved by
  $\mS^{\cJ}[-]\sm_{\mS^{\cJ}[D(v_1)]} \ell_p$.  Using~\cite[Lemma
  14.3]{Sagave-S_diagram}, the composite map can be identified with the map
  \[ \textstyle\coprod_{0\leq i\leq p-2} \Sigma^{2i}\ell_{p} \simeq
  \textstyle\coprod_{0\leq i\leq p-2}\ell_{p} \sm (F_1S^3)^{\sm i} \to
  ku_{p}\] induced by multiplication with iterated powers of the map
  $F_1S^3 \to ku_{p}$ from the free symmetric spectrum $F_1S^3 \simeq
  \Sigma^2 \mS$ specified by $u$. It is a stable equivalence since
  $\pi_*(\iota_{p})$ sends $v_1$ to $u^{p-1}$.

  It remains to verify that the middle map
  in~\eqref{eq:ku-decomposition} is a stable equivalence. By
  construction and Lemma~\ref{lem:Cx-to-Dx_on-hJ}, the map $Q_{\geq
    0}S^0 \simeq D(v_1)_{h\cJ} \to D(u)_{h\cJ}\simeq Q_{\geq 0}S^0$ is
  multiplication with $(p-1)$. Since the multiplication with the image
  of the generator $\id_{(\bld{3},\bld{1})}$ of $C(u)$ in $D(u)$ is a
  weak equivalence, it follows that $(E')_{h\cJ} \to D(u)_{h\cJ}$ is a
  $\pi_0$-isomorphism and a $\pi_i(-)\tensor \mZ_{(p)}$-isomorphism
  for $i\geq 1$.  Using the homotopy fiber sequence of~\cite[Lemma
  4.2]{Sagave-S_diagram} and the fact that $E'$ and $D(u)$ are
  positive fibrant, a five lemma argument implies that for all objects
  $(\bld{n_1},\bld{n_2})$ of $\cJ$ with $n_1 \geq 1$, 
  \begin{equation}\label{eq:Eprime-Du-components}
    E'(\bld{n_1},\bld{n_2}) \to  D(u)(\bld{n_1},\bld{n_2})
  \end{equation}
  is a $\pi_0$-isomorphism and a $\pi_i(-)\tensor
  \mZ_{(p)}$-isomorphism for $i\geq 1$.  Since the map $D(u)_{h\cJ} \to B\cJ$
  is a map of associative simplicial monoids, the
  non-empty homotopy cofibers have the homotopy types of 
  simplicial monoids. Hence the components of the
  $D(u)(\bld{n_1},\bld{n_2})$ for $n_1 \geq 1$ are simple spaces. The
  same argument applies to $D(v_1)_{h\cJ}\to B\cJ$, and the commutative
  diagram
  \[\xymatrix@-1pc{
    D(v_1)_{h\cJ} \ar[d]& \ar[l]^-{\sim}_-{\textrm{pr}}
    (\cJ((\bld{1},\bld{3}),-)^{\boxtimes i})_{h\cJ} \times
    D(v_1)_{h\cJ} \ar[r]^-{\sim}\ar[d] &
    (\cJ((\bld{1},\bld{3}),-)^{\boxtimes i}\boxtimes D(v_1))_{h\cJ} \ar[d]\\
    B\cJ & \ar[l]^-{\sim}_-{\textrm{pr}} (\cJ((\bld{1},\bld{3}),-)^{\boxtimes
      i})_{h\cJ} \times B\cJ \ar[r]^-{\sim} & B\cJ }
  \] 
  implies that the components of $E'(\bld{n_1},\bld{n_2})$ for
  $n_1\geq 1$ are simple spaces. It follows that the
  map~\eqref{eq:Eprime-Du-components} is a map of spaces with
  nilpotent path components and hence a
  $H_*(-;\mZ_{(p)})$-equivalence.

  The cofibrancy assumptions on $D(v_1)$ and $D(u)$ imply that the
  $\Sigma_{n_2}$-actions on $E'(\bld{n_1},\bld{n_2})$ and
  $D(u)(\bld{n_1},\bld{n_2})$ are free for all $n_2 \geq 1$
  (see~\cite[Appendix A]{Rognes-S-S_logTHH} for details). Hence the
  explicit description of $\mS^{\cJ}[-]$
  in~\cite[(4.5)]{Sagave-S_diagram} and an application of the
  Cartan-Leray spectral sequence show that $\mS^{\cJ}[E'] \to
  \mS^{\cJ}[D(u)]$ is a $H_*(-;\mZ_{(p)})$-isomorphism in positive
  levels. Since the assembly map $X \sm H\mZ_{(p)} \to \mZ_{(p)}[X]$ is a
  $\pi_*$-isomorphism for symmetric spectra $X$~\cite[II.\S
  6]{Schwede_SymSp}, it follows that the map 
  \[\mS^{\cJ}[E']\sm H\mZ_{(p)} \to
  \mS^{\cJ}[D(u)]\sm H\mZ_{(p)}\] is a stable equivalence. Our
  cofibrancy assumptions on $D(u)$ and $D(v_1)$ imply that the
  symmetric spectra $\mS^{\cJ}[E'] $ and $ \mS^{\cJ}[D(u)]$ are
  flat. Hence the above smash products have the homotopy types of the
  derived smash products, and it follows that $\mS^{\cJ}[E'] \to
  \mS^{\cJ}[D(u)]$ is a $H\mZ_{(p)}$-local equivalence. Because $E'$
  and $D(u)$ are concentrated in non-negative $\cJ$-space degrees, a
  cell induction argument shows that the spectra $\mS^{\cJ}[E']$ and
  $\mS^{\cJ}[D(u)] $ are connective. It follows that $\mS^{\cJ}[E']
  \to \mS^{\cJ}[D(u)]$ is a $p$-local equivalence. After cobase change
  along $\mS^{\cJ}[D(v_1)] \to \ell_p$, it becomes a $p$-local
  equivalence of $\ell_p$-modules and hence a stable equivalence.  The
  same arguments apply in the $p$-local case.
\end{proof}
\begin{remark}
  The $p$-complete complex $K$-theory spectrum $ku_p$ admits other
  interesting graded pre-log structures. Let $p\colon S^1 \to
  (ku_p)_1$ and $u\colon S^3 \to (ku_p)_1$ represent the homotopy
  classes $p\in\pi_0(ku_p)$ and $u\in\pi_2(ku_p)$. As explained in
  Remark~\ref{rem:degree-zero-case}, we get a pre-log structure $D(p)$
  even though $p$ has degree $0$. Using that $\boxtimes$ is the
  coproduct in $\cC\cS^{\cJ}$, we obtain a pre-log structure
  $D(p)\boxtimes D(u)\to \Omega^{\cJ}(ku_p)$ on $ku_p$.  The resulting
  pre-log ring spectrum $(ku_p, D(p)\boxtimes D(u))$ has a canonical
  map to $ku_p[1/p,1/u]$ with its trivial log structure.  In view of
  Rognes' discussion in~\cite[\S 1.9]{Rognes_TLS}, $(ku_p,
  D(p)\boxtimes D(u))$ is a candidate for a hypothetical object known
  as the \emph{fraction field of topological $K$-theory}. Its
  existence is supported by algebraic $K$-theory computations of
  Ausoni-Rognes~\cite{Ausoni-R_fraction-field}. This example
  emphasizes that it is potentially interesting to develop an algebraic
  $K$-theory of graded log ring spectra. We intend to pursue this in a
  later paper.
\end{remark}

\section{Logarithmic topological Andr\'{e}-Quillen homology}\label{sec:log-taq} 
In this section we define graded log derivations and graded log
topological Andr\'{e}-Quillen homology. To a large extent, this works
analogous to the log $\TAQ$ for $\cI$-space log ring spectra developed
by Rognes in~\cite[\S 10 and \S 11]{Rognes_TLS}. Thus our presentation
may also serve as a review of Rognes constructions. However, there are
also some differences: The fact that commutative $\cJ$-space monoids
do not admit a zero object requires some extra care when dealing with
square zero extensions, derivations, and their corepresenting
objects. To make this section more readable, we have deferred the
proof of the necessary technical results about commutative $\cJ$-space
monoids to Section~\ref{sec:CSJ-derivations}.

\subsection{Logarithmic derivations}\label{subsec:log-derivations}
Throughout this section let $A$ be a positive fibrant commutative
symmetric ring spectrum and let $X$ be a left $A$-module spectrum.

Because $A$ is commutative, we may also view $X$ as a right
$A$-module.  The \emph{square zero extension} of $A$ by $X$ is the
commutative $A$-algebra $A\wdg X$ with multiplication
\[ (A\wdg X) \sm (A\wdg X) \iso (A\sm A) \wdg (A\sm X) \wdg (X \sm A) \wdg (X
\sm X) \to A \wdg X\] induced by $A\sm A \to A$, the (left
and right) $A$-module structure on $X$, and the trivial map $X \sm X
\to *$. The map $X \to *$ induces an augmentation $A\wdg X \to A$.

Since $A\wdg X$ is clearly not fibrant and we will often consider maps
into it, we fix a notation for a fibrant replacement:
\begin{definition}
  We write $A\wdg_{\!f\!}X$ for the positive fibrant replacement of
  $A\wdg X$ in commutative symmetric ring spectra which is defined by the
  factorization
  \[ \xymatrix@1{A\wdg X \; \ar@{>->}^{\sim}[r]& \; A\wdgfib X\;
    \ar@{->>}[r] &\; A}. \] Its augmentation is denoted by
  $\varepsilon_X \colon A\wdgfib X \to A$.
\end{definition}

Recall that $U^{\cJ}=\cJ((\bld{0},\bld{0}),-)$ denotes the monoidal
unit of $(\cS^{\cJ},\boxtimes)$. It is the initial commutative
$\cJ$-space monoid.

\begin{definition}\label{def:one-plus-X-J}
Let the commutative $\cJ$-space monoid $(1+X)^{\cJ}$ be a cofibrant
replacement of the pullback of 
\[ \xymatrix@1{ U^{\cJ} \ar[r]& \GLoneJof(A) & \GLoneJof(A\wdgfib X)
  \ar[l]_-{(\varepsilon_X)_*}}.\]
\end{definition}
The notation is chosen in analogy to the subgroup of elements of the
form $1+x$ in the units of the square zero extension of an ordinary
ring by a module $X$. Since $(1+X)^{\cJ}$ is augmented over $U^{\cJ}$,
it is concentrated in $\cJ$-space degree $0$ because $U^{\cJ}$ is. In
other words, $(1+X)^{\cJ}(\bld{m_1},\bld{m_2})=\emptyset$ unless
$m_2-m_1 = 0$. Although this may look wrong at the first sight, the
following results show that $(1+X)^{\cJ}$ captures the desired portion
of the units of the square zero extension.

Appealing again to the situation in ordinary algebra, we recall that
the units of a square zero decompose into the units of the ring and
the additive group of the module. The following two results about $
(1+X)^{\cJ}$ may be viewed as a homotopical counterpart of this.

Since $\boxtimes$ is the coproduct in $\cC\cS^{\cJ}$, we get a
canonical map
\[
\xymatrix@1{\GLoneJof(A) \boxtimes (1+X)^{\cJ} \to \GLoneJof(A\wdgfib X)}.
\]
\begin{lemma}\label{lem:GLoneJof_AwdgX_decomposition}
This map is a $\cJ$-equivalence. 
\end{lemma}
\begin{proof}
Applying $(-)_{h\cJ}$, the induced map fits into commutative square
\[\xymatrix@-1pc{
  ((1+X)^{\cJ})_{h\cJ} \ar[r] \ar[d] & (\GLoneJof(A) \boxtimes (1+X)^{\cJ})_{h\cJ} \ar[r] \ar[d] & \GLoneJof(A)_{h\cJ}\ar[d] \\
  ((1+X)^{\cJ})_{h\cJ} \ar[r] & \GLoneJof(A\wdgfib X)_{h\cJ} \ar[r]&
  \GLoneJof(A)_{h\cJ}}.\] The lower sequence is a homotopy fiber sequence
by~\cite[Corollary~11.4]{Sagave-S_diagram}, and
Lemma~\ref{lem:monoidal-structure-map-equiv} implies that the upper
map is a homotopy fiber sequence as well. Hence the claim follows from
the long exact sequence of homotopy groups since all spaces involved
are grouplike simplicial monoids.
\end{proof}
The underlying spectrum of a module spectrum plays the role of the
additive group of the module spectrum. We will prove the following
proposition about its relation to $(1+X)^{\cJ}$ in
Section~\ref{sec:CSJ-derivations}:
\begin{proposition}\label{prop:identification_1plusX-J}
  The spectrum associated with the $\Gamma$-space
  $\gamma((1+X)^{\cJ})$ is stably equivalent to the connective cover
  of the underlying spectrum of the $A$-module $X$.
\end{proposition}

We now explain how to form square zero extensions in pre-log ring
spectra:
\begin{construction}\label{cons:MplusXJ}
  Let $(M,\alpha)$ be a graded pre-log structure on $A$. The universal
  property of the coproduct $\boxtimes$ in $\cC\cS^{\cJ}$ and the maps
  \[\alpha \colon M \to \Omega^{\cJ}(A\wdgfib X)\quad \text{ and }
  \quad (1+X)^{\cJ} \to \GLoneJof(A\wdgfib X) \to
  \Omega^{\cJ}(A\wdgfib X)\] induce a pre-log structure $ M
  \boxtimes (1+X)^{\cJ} \to \Omega^{\cJ}(A\wdgfib X)$ on $A\wdgfib X$.
  Using the augmentation $(1+X)^{\cJ} \to U^{\cJ}$ and the isomorphism
  $M\boxtimes U^{\cJ}\iso M$, we obtain a map $M \boxtimes (1+X)^{\cJ}
  \to M$ such that the outer square in
  \begin{equation}\label{eq:def-MplusXJ}
    \xymatrix@-1pc{
      M \boxtimes (1+X)^{\cJ} \ar[rrr] \ar[dd] \ar@{>->}[dr]^{\sim}& & \qquad&\Omega^{\cJ}(A\wdgfib X)\ar@{->>}[dd]\\
      & (M+X)^{\cJ} \ar@{->>}[dl] \ar@{-->}[urr] \\
      M \ar[rrr]& & &\Omega^{\cJ}(A) }
  \end{equation} 
  commutes. We define $(M+X)^{\cJ}$ by the indicated factorization in
  the positive $\cJ$-model structure. Since $A\wdg_f X \to A$ is a
  fibration by assumption, the dotted arrow exists and defines the
  pre-log ring spectrum $(A\wdg_f X, (M+X)^{\cJ})$. Having ensured
  that $(M+X)^{\cJ} \to M$ is a fibration will turn out useful when
  considering maps into $(A\wdg_f X, (M+X)^{\cJ})$.
\end{construction}
  The diagram induces a sequence of maps $(A,M) \to (A\wdgfib
  X,(M+X)^{\cJ}) \to (A,M)$ making $(A,(M+X)^{\cJ})$ a graded pre-log
  ring spectrum under and over $(A,M)$. Using this, we can state the
  homotopical counterpart of the algebraic notion of log derivations
  outlined in the introduction:
\begin{definition}
Let $(R,P) \to (A,M)$ be a map of graded pre-log ring spectra and let $X$ be an $A$-module. 
A \emph{graded log derivation} of $(A,M)$ over $(R,P)$ with values in $X$ is a map 
\[ (d,d^{\flat}) \colon (A,M) \to (A\wdgfib
  X,(M+X)^{\cJ})\] 
under $(R,P)$ and over $(A,M)$. That is, $(d,d^{\flat})$ is a map such that the diagram
\[\xymatrix@-1pc{
  (R,P) \ar[rr] \ar[d] && (A\wdgfib
  X,(M+X)^{\cJ}) \ar[d] \\
  (A,M) \ar[rr] \ar@{-->}[urr]^(.4){(d,d^{\flat})}&& (A,M) }\] commutes.
\end{definition}
We will use this notion to motivate and justify the definition of
logarithmic topological Andr\'{e}-Quillen homology below. 

For later use we note two more properties of the induced pre-log
structure on the square zero extension:
\begin{lemma}
  The logification $(M,\alpha) \to (M^a,\alpha^a)$ induces a weak
  equivalence \begin{equation}\label{eq:logification-and-square-zero}(A\wdgfib X,((M+X)^{\cJ})^a) \to(A\wdgfib
  X,(M^a+X)^{\cJ}).\end{equation} If $(A,M)$ is a graded log ring spectrum, then
  so is $(A\wdgfib X,(M+X)^{\cJ})$.
\end{lemma}
\begin{proof}
  It follows from the lifting axioms that $(M,\alpha) \to
  (M^a,\alpha^a)$ extends to the map of pre-log ring
  spectra~\eqref{eq:logification-and-square-zero}. Since the log
  condition and the logification are invariant under weak
  equivalences, it is enough to check the claim for the pre-log
  structures defined by $M \boxtimes (1+X)^{\cJ}$ and $M^{a} \boxtimes
  (1+X)^{\cJ}$.

  Let $(M,\alpha)$ be a pre-log structure. In the notation of
  Lemma~\ref{lem:logification}, $M$ decomposes as the part
  $\widetilde{M}$ that maps to the units $\GLoneJof(A)$ of $A$ and its
  complement $\widehat{M}$. Since $(1+X)^{\cJ}$ is grouplike, it is
  easy to see that
  \[ (\widetilde{M} \boxtimes (1+X)^{\cJ}) \textstyle\coprod
  (\widehat{M} \boxtimes (1+X)^{\cJ}) \to \Omega^{\cJ}(A\wdgfib X)\]
  is the corresponding decomposition of $ M \boxtimes (1+X)^{\cJ}$
  over $\Omega^{\cJ}(A\wdgfib X)$. Using that $(1+X)^{\cJ}$ is flat
  and that the $\boxtimes$-product with a flat $\cJ$-space preserves
  $\cJ$-equivalences (\cite[Proposition 8.2]{Sagave-S_diagram}),
  Lemma~\ref{lem:GLoneJof_AwdgX_decomposition} shows the second claim
  of the lemma.

  Defining $G$ as in the construction of the logification of
  $(M,\alpha)$, the same arguments show that the factorization
  \[\xymatrix@1{\widetilde{M} \boxtimes (1+X)^{\cJ} \; \ar@{>->}[r] &
    \; G \boxtimes (1+X)^{\cJ}\; \ar[r]^{\sim} & \; \GLoneJof(A\wdgfib
    X)} \] may be used to form the pushout giving the logification of
  $(M+X)^{\cJ}$. With this choice for the cofibrant
  replacement, the pushout is $M^a \boxtimes (1+X)^{\cJ}$.
\end{proof}
\begin{lemma}\label{lem:hty-pullback-with-square-zero}
The commutative square
\[\xymatrix@-1pc{
(M+X)^{\cJ} \ar[r] \ar[d] & (M^a+X)^{\cJ} \ar[d] \\
M \ar[r] & M^a
}
\]
induced by the logification is a homotopy pullback. 
\end{lemma}
\begin{proof}
  This follows by combining the description of $ (M^a+X)^{\cJ}$ in the
  last lemma, Lemma~\ref{lem:monoidal-structure-map-equiv}, and the
  fact that homotopy pullbacks in $\cJ$-spaces are detected by
  $(-)_{h\cJ}$ (see~\cite[Corollary
  11.4]{Sagave-S_diagram}).
\end{proof}

\subsection{Spaces of maps between graded pre-log ring spectra} When working
with mapping spaces, one often has to ensure a cofibrant domain and a
fibrant codomain. To implement this in the case at hand, we use the
following ``projective'' model category structure on graded pre-log ring
spectra whose proof easily follows from the fact that
$(\mS^{\cJ}[-],\Omega^{\cJ})$ is a Quillen adjunction:
\begin{lemma}
  The category of graded pre-log ring spectra admits a model structure in
  which $(f,f^{\flat}) \colon (A,M) \to (B,N)$ is a weak equivalence
  (or fibration) if $f$ is a stable equivalence (or stable positive
  fibration) in $\cC\Spsym$ and $f^{\flat}$ is a $\cJ$-equivalence (or
  positive $\cJ$-fibration) in $\cC\cS^{\cJ}$.\qed
\end{lemma}
So the weak equivalences are those of
Definition~\ref{def:graded-pre-log}, and $(A,M)$ is cofibrant if and
only if $M$ is cofibrant and $\mS^{\cJ}[M] \to A$ is a cofibration.

Both commutative symmetric ring spectra and commutative $\cJ$-space
monoids are simplicial model categories. We denote their mapping
spaces by $\Map_{\cC\Spsym}$ and $\Map_{\cC\cS^{\cJ}}$. These mapping
space can be defined using the tensor with the cosimplicial simplicial
set $[n] \mapsto \Delta^{n}$, and this tensor is in turn defined via the
coproduct over $\Delta^{n}_k$ and the realization. In particular, the left adjoint
$\mS^{\cJ}[-]$ commutes with the tensor and induces a map of spaces
$\Map_{\cC\cS^{\cJ}}(M,N)\to
\Map_{\cC\Spsym}(\mS^{\cJ}[M],\mS^{\cJ}[N])$.

If $(A,M)$ and $(B,N)$ are graded pre-log ring spectra, then their structure maps
induce a diagram 
\begin{equation}\label{eq:def-mapping-space}
\Map_{\cC\cS^{\cJ}}(A,B) \to \Map_{\cC\cS^{\cJ}}(\mS^{\cJ}[M],B) \ot \Map_{\cC\cS^{\cJ}}(M,N)
\end{equation}
\begin{definition}
The space of maps $\Map_{\cP}((A,M),(B,N))$ between pre-log ring spectra $(A,M)$ and $(B,N)$
is the pullback of diagram~\eqref{eq:def-mapping-space}.
\end{definition}
\begin{corollary}
  If $(A,M)$ is cofibrant and $(B,N)$ is fibrant, then the pullback
  of~\eqref{eq:def-mapping-space} captures the homotopy type of its
  homotopy pullback. \qed
\end{corollary}
Hence the mapping space $\Map_{\cP}$ is invariant under weak
equivalences in both variables if the objects are sufficiently
cofibrant and fibrant.

If $\cC$ is a model category and $E \to G$ is a map in $\cC$, then
there is a canonical model structure on the category $\cC^{E}_{G}$ of
objects under $E$ and over $G$: Objects in $\cC^{E}_{G}$ are
factorizations $E \to F \to G$, and a map from $E \to F \to G$ to $E \to F'
\to G$ is a weak equivalence, cofibration, or fibration if and only if
its projection to $F \to F'$ is. If $\cC$ is in addition a simplicial
model category, taking iterated pullbacks of diagrams induced by the
augmentation and the coaugmentation of the domain and the codomain
defines mapping spaces for $\cC^{E}_{G}$.

\begin{definition}
Let $(R,P) \to (A,M)$ be a cofibration of graded pre-log ring spectra and 
let $X$ be a fibrant $A$-module. The \emph{space of graded log derivations} of $(A,M)$
over $(R,P)$ with values in $X$ is the mapping space
\[ \Der_{(R,P)}((A,M),X) = \Map^{(R,P)}_{(A,M)}((A,M),(A\wdgfib
X,(M+X)^{\cJ}))\] This is homotopically meaningful because $ (A\wdgfib
X,(M+X)^{\cJ})$ is fibrant.
\end{definition}
\subsection{Construction of graded log \texorpdfstring{$\TAQ$}{TAQ}}
We will now explain how the space $\Der_{(R,P)}((A,M),X)$ can be corepresented by
an $A$-module. 

We start by decomposing this space as a homotopy pullback.  For this
we will write $\Map^P_M(-,-)$ for the mapping space in
$(\cC\cS^{\cJ})^P_M$ and $\Map^R_A(-,-)$ for the mapping space in
$(\cC\Spsym)^R_A$.
Then the commutative squares 
\[\xymatrix@-1pc{
\mS^{\cJ}[(M+X)^{\cJ}] \ar[r] \ar[d]& A\wdgfib X \ar[d]  \\
\mS^{\cJ}[M] \ar[r]& A}
\qquad \text{and} \qquad 
\xymatrix@-1pc{\mS^{\cJ}[P] \ar[r]\ar[d]& R\ar[d]\\ \mS^{\cJ}[M] \ar[r] & A}
\]
induced by the maps $(A\wdgfib X,(M+X)^{\cJ}) \to
(A,M)$ and $(R,P) \to (A,M)$ give rise to a commutative square
\begin{equation}\label{eq:log-der-decomposition}\xymatrix@-1pc{
\Der_{(R,P)}((A,M),X) \ar[r] \ar[d] & \Map^R_A(A,A\wdgfib X) \ar[d] \\
\Map^P_M(M,(M+X)^{\cJ}) \ar[r]   & \Map^{\mS^{\cJ}[P]}_A(\mS^{\cJ}[M],A\wdgfib X).
}
\end{equation}
The description of the mapping space in~\eqref{eq:def-mapping-space} implies
\begin{lemma}\label{lem:log-der-decomposition}
  The square~\eqref{eq:log-der-decomposition} is homotopy
  cartesian. \qed
\end{lemma}
By definition, the upper right hand corner
of~\eqref{eq:log-der-decomposition} is the space $\Der_{R}(A,X) $ of
$R$-algebra derivations of $A$ with values in $X$. The corepresenting
object has been constructed by Basterra~\cite{Basterra_Andre-Quillen}
and is known as the \emph{topological Andr\'{e}-Quillen homology}
$\TAQ^R(A)$ of $R \to A$. It has been extensively studied in the
literature.

We briefly recall the construction of $\TAQ$
from~\cite{Basterra_Andre-Quillen}: Using the left derived product
over $R$, the map $R \to A$ gives rise to an augmented $A$-algebra
$A\sm_R^{\mathbb{L}}A$.  The augmentation ideal $I_A$ may be viewed as
a functor from augmented $A$-algebras to the category non-unital
$A$-algebras $\cN_A$. With suitable model structures, it participates
as the right adjoint in a Quillen equivalence. Evaluating its right
derived functor $I_A^{\mathbb R}$ to $A\sm_R^{\mathbb{L}}A$ provides a
non-unital $A$-algebra. Applying the left derived functor of the
indecomposables $Q_A\colon \cN_A \to A\text{-}\Mod$ to this non-unital
$A$-algebra defines the $A$-module $\TAQ^R(A)$.  The fact that
Basterra uses $S$-modules in the sense of~\cite{EKMM} rather than
symmetric spectra in her description may be dealt with by either
mimicking her approach in symmetric spectra, or using Schwede's
equivalence between symmetric spectra and
$S$-modules~\cite{Schwede_S-modules-Spsym} to go back and forth
between the two setups.

As indicated above, we shall mostly need the following result about
$\TAQ$:
\begin{proposition}\cite[Proposition 3.2]{Basterra_Andre-Quillen}\label{prop:TAQ-correpresents-Der}
  The space $Der_R(A,X)$ is naturally weakly equivalent to the space
  of maps of $A$-module spectra $\Map_{A\text{-}\Mod}(\TAQ^R(A),X)$.
\end{proposition}
To study the lower right hand corner
of~\eqref{eq:log-der-decomposition}, we first observe that there is a
dotted arrow making the following diagram commutative:
\[\xymatrix@-1pc{
  && \mS^{\cJ}[M] \wdg X \ar[rr] \ar@{>->}[d]_{\sim}&& A\wdg X  \ar@{>->}[d]^{\sim}\\
  \mS^{\cJ}[P] \ar[rr]\ar[d]&& \mS^{\cJ}[M]\wdgfib X \ar@{-->}[rr]\ar@{->>}[d]&& A\wdgfib X\ar@{->>}[d] \\
  \mS^{\cJ}[M] \ar[rr]^{=}&& \mS^{\cJ}[M] \ar[rr]&& A }\] The fact
that the lower right hand square is homotopy cartesian easily implies
that there is a weak equivalence \[\Der_{\mS^{\cJ}[P]}(\mS^{\cJ}[M],X) \to
\Map^{\mS^{\cJ}[P]}_A(\mS^{\cJ}[M],A\wdgfib X).\]
So the previous proposition implies
\begin{corollary}\label{cor:correprsenting_one_corner}
  The space $\Map^{\mS^{\cJ}[P]}_A(\mS^{\cJ}[M],A\wdgfib X)$ is
  naturally weakly equivalent to mapping space of $A$-module spectra
  $\Map_{A\text{-}\Mod}(A\sm_{\mS^{\cJ}[M]}\TAQ^{\mS^{\cJ}[P]}(\mS^{\cJ}[M]),
  X)$.
\end{corollary}
The lower left hand corner $\Map^P_M(M,(M+X)^{\cJ})$ of~\eqref{eq:log-der-decomposition} may be
interpreted as the space of commutative $\cJ$-space monoid derivations,
i.e., as the space $P$-algebra derivations of $M$ with values in the
grouplike commutative $\cJ$-space monoid $(1+X)^{\cJ}$. We will prove
the following result in Section~\ref{sec:CSJ-derivations}:
\begin{proposition}\label{prop:CSJ-derivations}
  The space $\Map^P_M(M,(M+X)^{\cJ})$ is naturally weakly equivalent
  to the space of maps of $A$-module spectra
  $\Map_{A\text{-}\Mod}(A\sm (\gamma(M)/\gamma(P)), X)$.
\end{proposition}
In the proposition, $\gamma(M)/\gamma(P)$ is the homotopy cofiber of
the map of $\Gamma$-spaces induced by $P \to M$, and $A\sm
(\gamma(M)/\gamma(P))$ is the $A$-module spectrum obtained from the
spectrum associated with $\gamma(M)/\gamma(P)$ by extension of scalars
along $\mS \to A$.

As functors in $X$, the mapping spaces in all but the upper left hand
corner of the square~\eqref{eq:log-der-decomposition} are
corepresented by $A$-modules. Hence we obtain the following maps
between the corepresenting objects:
\begin{equation}\label{eq:defining-graded-log-taq}
A \sm (\gamma(M)/\gamma(P)) \ot A\sm_{\mS^{\cJ}[M]}\TAQ^{\mS^{\cJ}[P]}(\mS^{\cJ}[M]) \to \TAQ^R(A)
\end{equation}
\begin{definition}\label{def:graded-log-TAQ}
  Let $(R,P) \to (A,M)$ be a map of graded pre-log ring spectra. The
  \emph{graded log topological Andr\'{e}-Quillen homology}
  $\TAQ^{(R,P)}(A,M)$ of $(A,M)$ over $(R,P)$ is the $A$-module given
  by the homotopy pushout of \eqref{eq:defining-graded-log-taq}.
\end{definition}
We call this \emph{log} $\TAQ$ rather than \emph{pre-log} $\TAQ$
because it is invariant under logification. We prove this in
Corollary~\ref{cor:log-TAQ-log-inv} below. The following result is the
main motivation behind the definition of $\TAQ^{(R,P)}(A,M)$:
\begin{proposition}\label{prop:logTAQ-correpresents-Der}
  Let $(R,P) \to (A,M)$ be a cofibration of graded pre-log ring
  spectra, let $(A,M)$ be fibrant, and let $X$ be a fibrant
  $A$-module. There is a weak equivalence between
  $\Map_{A\text{-}\Mod}(\TAQ^{(R,P)}(A,M), X)$ and
  $\Der_{(R,P)}((A,M),X)$.
\end{proposition}
\begin{proof}
  Since $\Map_{A\text{-}\Mod}(-,X)$ maps homotopy pushouts to homotopy
  pullbacks, this follows from the definition of graded log $\TAQ$,
  Proposition~\ref{prop:TAQ-correpresents-Der},
  Corollary~\ref{cor:correprsenting_one_corner}, and
  Proposition~\ref{prop:CSJ-derivations}.
\end{proof}
\begin{definition}\label{def:graded-log-etale}
  A map of graded pre-log ring spectra $(R,P)\to (A,M)$ is
  \emph{formally graded log \'{e}tale} if the $A$-module
  $\TAQ^{(R,P)}(A,M)$ is contractible.
\end{definition}
Examples of formally graded log \'{e}tale extensions will be given in the next section. 

\section{Log \'{e}tale extensions of \texorpdfstring{$K$}{K}-theory
  spectra}\label{sec:log-etale-maps}
Let $p$ be an odd prime. Recall from
Proposition~\ref{prop:log-maps-ell-ku} that the inclusion of the
$p$-complete Adams summand into the $p$-complete connective $K$-theory spectrum 
$\ell_p \to ku_{p}$ can be extended to a map of graded log ring
spectra
\[
(\ell_p, i_*\!\GLoneJof(L_p)) \to (ku_{p}, i_*\!\GLoneJof(KU_{p})).
\]
The graded log-structures are the direct image log structures of the
respective periodic theories. The following theorem is one of the main
results of this paper. As discussed in the introduction, it confirms
that $\ell_p \to ku_{p}$ should be viewed as a tamely ramified
extension of ring spectra. 
\begin{theorem}\label{thm:ell-to-k-formally-etale}
  The map $(\ell_p, i_*\!\GLoneJof(L_p)) \to (ku_{p},
  i_*\!\GLoneJof(KU_{p}))$ is formally graded log \'{e}tale. That is,
  the graded log topological Andr\'{e}-Quillen homology spectrum
  \[ \TAQ^{(\ell_p, i_*\!\GLoneJof(L_p))}(ku_{p},
  i_*\!\GLoneJof(KU_{p}))\] is contractible. The same holds in the
  $p$-local case.
\end{theorem}

The proof of Theorem~\ref{thm:ell-to-k-formally-etale} will be given
at the end of this section. The first step towards its proof is the
following graded analogue of~\cite[Lemma 11.25]{Rognes_TLS}. It gives a
useful criterion for showing that a map is formally graded log
\'{e}tale:
\begin{lemma}
Let $(R,P) \to (A,M)$ be a map of graded pre-log ring spectra, and let $C$
be the homotopy pushout of $R \ot \mS^{\cJ}[P] \to \mS^{\cJ}[M]$. Then there is 
a homotopy cofiber sequence
\begin{equation}\label{eq:hocofib-sequence-graded-log-TAQ}
A \sm (\gamma(M)/\gamma(P)) \to \TAQ^{(R,P)}(A,M) \to \TAQ^C(A)
\end{equation}
of $A$-module spectra.  
\end{lemma}
\begin{proof}
  The proof is completely analogous to the proof of~\cite[Lemma
  11.25]{Rognes_TLS}: The defining homotopy
  pushout~\eqref{eq:defining-graded-log-taq} shows that the homotopy
  cofiber of the left map in~\eqref{eq:hocofib-sequence-graded-log-TAQ} is equivalent to the
  homotopy cofiber of
  \[A\sm_{\mS^{\cJ}[M]}\TAQ^{\mS^{\cJ}[P]}(\mS^{\cJ}[M]) \to \TAQ^R(A)\]
  Flat base change and the transitivity sequence for $\TAQ$~\cite[\S
  4]{Basterra_Andre-Quillen} allow to identify this homotopy cofiber
  with $\TAQ^{C}(A)$.
\end{proof}
\begin{corollary}\label{cor:etaleness-criterion}
  The map $(R,P) \to (A,M)$ is formally graded log \'{e}tale if the
  map $\gamma(P) \to \gamma(A)$ is an $A$-homology equivalence and
  $\TAQ^C(A)$ is contractible.\qed
\end{corollary}

The next aim is to show that graded log $\TAQ$ is invariant under logification. 
For this we begin with
\begin{lemma}\label{lem:map-into-log-logification-inv}
Let $(A,M)$ be a pre-log ring spectrum and let $(B,N)$ be a log ring spectrum.
The logification $(A,M) \to (A,M^a)$ induces a weak equivalence of mapping
spaces $\Map_{\cP}((A,M^a),(B,N)) \to \Map_{\cP}((A,M),(B,N))$
\end{lemma}
\begin{proof}
  We may assume that $N \to \Omega^{\cJ}(B)$ is a positive
  fibration. By an adjunction argument~\cite[Lemma
  9.3.6]{Hirschhorn_model}, it is enough to show that the pushout
  product map of
\begin{equation}\label{eq:pushout-product-of-logification} ((A,M) \to (A,M^a)) \tensor (\partial \Delta^n \to \Delta^n) \end{equation} has
the left lifting property with respect to $(B,N) \to *$.

Let $K \to L$ be a cofibration in $\cC\cS^{\cJ}$ with $L$
grouplike. The log condition on $(B,N)$ implies that $N \to
\Omega^{\cJ}(B)$ has the right lifting property with respect to $K \to
L$.  Using the adjunction $(\mS^{\cJ}[-],\Omega^{\cJ})$, this shows
that $(B,N) \to *$ has the right lifting property with respect to
$(\mS^{\cJ}[L],K)\to(\mS^{\cJ}[L],L)$.  The logification $(A,M) \to
(A,M^a)$ can be viewed as the map from the pushout to the lower right
corner in the diagram
\[\xymatrix@-1pc{
  (\mS^{\cJ}[G],\alpha^{-1}(\GLoneJof(A))) \ar[r] \ar[d] & (A,M) \ar[d]\\
  (\mS^{\cJ}[G],G) \ar[r] & (A,M^a).  }\] This shows that the pushout
product map of~\eqref{eq:pushout-product-of-logification} can be
obtained as the cobase change of the pushout product map of
\begin{equation}\label{eq:pushout-product-of-mSJs} \left((\mS^{\cJ}[G],\alpha^{-1}(\GLoneJof(A))) \to
    (\mS^{\cJ}[G],G)\right) \tensor (\partial \Delta^n \to \Delta^n)
\end{equation}
along the map induced by
$\left((\mS^{\cJ}[G],\alpha^{-1}(\GLoneJof(A))) \to
  (A,M)\right)\tensor \Delta^n$.  So it is enough to show that this
pushout product map has the lifting property w.r.t.  $(B,N)\to *$. One
can check that it is the map of pre-log ring spectra
$(\mS^{\cJ}[L],K)\to(\mS^{\cJ}[L],L)$ associated with the map
\[ K = \alpha^{-1}(\GLoneJof(A)) \tensor \Delta^n \boxtimes_{
  \alpha^{-1}(\GLoneJof(A))\tensor \partial \Delta^n} G
\tensor \partial \Delta^n \to G\tensor \Delta^n = L.\] Since
$\cC\cS^{\cJ}$ is a simplicial model category, $K \to L$ is a
cofibration.  The $G\tensor \Delta^n$ is grouplike because $G$ is. So
the claim follows by the right lifting property for $(B,N)\to *$ established
above.
\end{proof}
\begin{remark}
  The argument in the proof of the last lemma can be used show that
  the (pre-fibrant) log ring spectra may be identified with the local
  objects of a left Bousfield localization of the category pre-log
  ring spectra. The fibrant replacement in the localization is then a
  model for the logification. We omit the details because this
  \emph{log model structure} is not needed in the present paper.
\end{remark}

We now look at the square
\begin{equation}\label{eq:logification-of-map}\xymatrix@-1pc{
    (R,P) \ar[r] \ar[d] & (R,P^a) \ar[d] \\
    (A,M) \ar[r] & (A,M^a) }\end{equation}
induced by the logification of $(R,P)$ and $(A,M)$. The following lemma
is the graded analogue of~\cite[Lemma 11.9]{Rognes_TLS}. In the lemma,
we implicitly assume that the vertical arrows in~\eqref{eq:logification-of-map}
are cofibrations of fibrant objects. (This can always be achieved up to 
weak equivalence.)

\begin{lemma}\label{lem:derivations-invariant-logification}
  Let $X$ be a fibrant $A$-module. The maps in the
  square~\eqref{eq:logification-of-map} induce weak equivalences
  \[\xymatrix@-1pc{
\Der_{(R,P)}((A,M),X) \ar[r]^-{\sim}& \Map^{(R,P)}_{(A,M^a)}((A,M),(A\wdgfib X,(M^a+X)^{\cJ})) \\
\Der_{(R,P^a)}((A,M^a),X)\ar[r]^-{\sim}&\Der_{(R,P)}((A,M^a),X). \ar[u]^-{\sim}
} \]
\end{lemma}
\begin{proof}
  The upper horizontal map is a weak equivalence by the homotopy
  pullback established in
  Lemma~\ref{lem:hty-pullback-with-square-zero}. The two other maps
  are weak equivalences by
  Lemma~\ref{lem:map-into-log-logification-inv}.
\end{proof}

Proposition~\ref{prop:logTAQ-correpresents-Der} and the last lemma
now easily imply the counterpart of~\cite[Corollary 11.23]{Rognes_TLS}:
\begin{corollary}\label{cor:log-TAQ-log-inv} The maps in~\eqref{eq:logification-of-map} induce a zig-zag of stable
  equivalences of $A$-modules between $\TAQ^{(R,P)}(A,M)$,
  $\TAQ^{(R,P)}(A,M^a)$, and $\TAQ^{(R,P^a)}(A,M^a)$.\qed
\end{corollary}

We have now developed all tools for the proof of the main theorem of
this section:
\begin{proof}[Proof of Theorem~\ref{thm:ell-to-k-formally-etale}]
  By Corollary~\ref{cor:log-TAQ-log-inv},
  Proposition~\ref{prop:log-maps-ell-ku}, and
  Theorem~\ref{thm:all-pre-log-on-A} it is enough to show that the map
  of pre-log ring spectra
  \[ (\ell_p, D(v_1)) \to (ku_{p}, D(u)) \] is formally graded log
  \'{e}tale. We use the criterion of
  Corollary~\ref{cor:etaleness-criterion}. Proposition~\ref{prop:ku-p-as-hty-pushout}
  implies that the map from the homotopy pushout $C$ of
  \[ \ell_p \ot \mS^{\cJ}[D(v_1)] \to \mS^{\cJ}[D(u)]\] to $ku_{p}$ is
  a stable equivalence. Hence $\TAQ^C(ku_{p})$ is contractible.  Using
  Lemma~\ref{lem:Cx-to-Dx_on-hJ}, it is easy to see that both
  $\gamma(D(v_1))$ and $\gamma(D(u))$ have the homotopy type of the
  sphere spectrum, and that the map $\gamma(D(v_1)) \to \gamma(D(u))$
  is the multiplication by $p-1$. Since $ku_{p}$ is $p$-complete, the
  induced map $A \sm \gamma(D(v_1)) \to A \sm \gamma(D(u))$ is a
  stable equivalence.  The same arguments apply in the $p$-local case.
\end{proof}

\section{Commutative \texorpdfstring{$\cJ$}{J}-space monoid derivations}\label{sec:CSJ-derivations}
In this section we give the proof of
Proposition~\ref{prop:identification_1plusX-J} and
Proposition~\ref{prop:CSJ-derivations}. These relate the space of
commutative $\cJ$-space monoid derivations $\Map^P_M(M,(M+X)^{\cJ})$ with the
$A$-module spectrum $X$ used to define $ (M+X)^{\cJ}$.

Let $A$ be a positive fibrant in $\cC\Spsym$ and let $X$ be an
$A$-module.  Similarly as the commutative $\cJ$-space monoid
$(1+X)^{\cJ}$ introduced in Definition~\ref{def:one-plus-X-J}, we
obtain commutative $\cI$-space monoid $(1+X)^{\cI}$ as cofibrant
replacement of the pullback of
\[ \xymatrix@1{ U^{\cI} \ar[r]& \GLoneIof(A) & \GLoneIof(A\wdgfib X)
  \ar[l]_-{(\varepsilon_X)_*}}.\] As indicated in
Section~\ref{sec:CSJ-group-compl} and explained in detail in~\cite[\S
3]{Sagave_spectra-of-units}, the commutative $\cJ$-space monoid
$(1+X)^{\cJ}$ gives rise to a $\Gamma$-space
$\gamma((1+X)^{\cJ})$. This functor $\gamma$ is in turn motivated by a
functor $\gamma\colon \cC\cS^{\cI} \to \GammaS$ which is (under a
different name) considered by Schlichtkrull
in~\cite{Schlichtkrull_units} and the author and Schlichtkrull
in~\cite{Sagave-S_group-compl}.  The latter functor provides a
$\Gamma$-space $\gamma((1+X)^{\cI})$ associated with $(1+X)^{\cI}$.
\begin{lemma}\label{lem:one-plus-X-J-I-reduction}
The $\Gamma$-spaces  $\gamma((1+X)^{\cI})$ and $\gamma((1+X)^{\cJ})$
are level equivalent. 
\end{lemma}
\begin{proof}
  The functors $\gamma$ send $\cI$- and $\cJ$-equivalences to level
  equivalences of $\Gamma$-spaces. Hence it is enough to show the claim
  for the actual pullbacks used to define $(1+X)^{\cI}$ and
  $(1+X)^{\cJ}$. In this proof, we denote them also by $(1+X)^{\cI}$
  and $(1+X)^{\cJ}$.

  By~\cite[Lemma 2.12]{Sagave_spectra-of-units}, the strong symmetric
  monoidal functor $\Delta\colon \cI \to \cJ$ sending $\bld{m}$ to
  $(\bld{m},\bld{m})$ has the property $\Delta^*\GLoneJof(A) \iso
  \GLoneIof(A)$. We obtain a commutative diagram of commutative $\cI$-space monoids
\[\xymatrix@-1pc{
  \Delta^*U^{\cJ} \ar[r]  & \Delta^*\GLoneJof(A)  & \Delta^*\GLoneJof(A\wdgfib X)  \ar[l] \\
  U^{\cI} \ar[r] \ar[u]& \GLoneIof(A)\ar[u]& \GLoneIof(A\wdgfib
  X)\ar[l]\ar[u].  }\] Using the natural transformation
$(\Delta^*(-))_{h\cI}\to (-)_{h\cJ}$ induced by $\Delta$, the fact
that $\Delta^*$ commutes with pullbacks and~\cite[Lemma
3.16]{Sagave_spectra-of-units} show that we get a map $c_X\colon
\gamma((1+X)^{\cI}) \to \gamma((1+X)^{\cJ})$.  To see that it is a
level equivalence, it is enough to that its evaluation at $1^+$ is a
weak equivalence since both $\Gamma$-spaces are special. The natural
transformation $(\Delta^*(-))_{h\cI}\to (-)_{h\cJ}$ provides a
commutative square
  \begin{equation}\label{eq:one-plus-X-J-I-reduction}\xymatrix@-1pc{
      (\GLoneIof(A\wdgfib X))_{h\cI} \ar[r]^-{\iso} \ar[d] & \left(\Delta^*\GLoneJof(A\wdgfib X)\right)_{h\cI}\ar[r] \ar[d] &(\GLoneJof(A\wdgfib X))_{h\cJ} \ar[d]\\
      (\GLoneIof(A))_{h\cI} \ar[r]^-{\iso} &
      \left(\Delta^*\GLoneJof(A)\right)_{h\cI}\ar[r]
      &(\GLoneJof(A))_{h\cJ}.  }
  \end{equation} 
  Since both $U^{\cI}_{h\cI}$ and $ U^{\cJ}_{h\cJ}$ are contractible,
  $c_X(1^+)$ will be a weak equivalence if the right hand square
  in~\eqref{eq:one-plus-X-J-I-reduction} is homotopy cartesian. This
  holds because the map of horizontal homotopy fibers in the outer
  square is a weak equivalence by~\cite[Proposition
  4.1]{Sagave_spectra-of-units}.
\end{proof}
\begin{proof}[Proof of Proposition~\ref{prop:identification_1plusX-J}]
  It is shown in~\cite[Lemma~11.2]{Rognes_TLS} that the
  $\Gamma$-space associated with the commutative $\cI$-space monoid
  $(1+X)^{\cI}$ is stably equivalent to the connective cover of the
  underlying spectrum of $X$. Using this, the $\cJ$-space version
  follows from Lemma~\ref{lem:one-plus-X-J-I-reduction}.
\end{proof}

We start to work towards the proof of
Proposition~\ref{prop:CSJ-derivations}. Recall from
Construction~\ref{cons:MplusXJ} that $(M+X)^{\cJ}$ is fibrant
replacement of $M \boxtimes (1+X)^{\cJ}$ relative to $M$. In
particular, it is an object under $P$ and over $M$. We will again use
the functor $\gamma \colon \cC\cS^{\cJ} \to \GammaS$
of~\eqref{eq:functor-gamma} and write $(\gamma(1+X)^{\cJ})^{\fib}$ for
a level fibrant replacement of the (very special) $\Gamma$-space
associated with $(1+X)^{\cJ}$. Since $\GammaS$ has a zero object,
$(1+X)^{\cJ}$ gives rise to a second $\Gamma$-space \[\gamma(P) \to
\gamma(M)\times (\gamma(1+X)^{\cJ})^{\fib} \xrightarrow{\text{pr}}
\gamma(M)\] under $\gamma(P)$ and over $\gamma(M)$.
\begin{lemma}\label{lem:gamma-M-plus-X-J-as-product}
  With respect to the maps from $\gamma(P)$ and to $\gamma(M)$
  specified above, $\gamma(M+X)^{\cJ}$ and $\gamma(M)\times
  (\gamma(1+X)^{\cJ})^{\fib} $ are level equivalent in
  $\gamma(P)\!\downarrow\!\GammaS \!\downarrow\!
  \gamma(M)$.\end{lemma}
\begin{proof}
  We write $N=(1+X)^{\cJ}$. A levelwise application of the
  map~\eqref{eq:monoidal-structure-map} provides a map of
  $\Gamma$-spaces $\gamma(M)\times\gamma(N) \to \gamma(M\boxtimes
  N)$. (See the proof of \cite[Lemma 7.22]{Sagave_spectra-of-units}
  for more details about this map.) By
  Lemma~\ref{lem:monoidal-structure-map-equiv}, the map is a level
  equivalence since $N$ is flat by construction.  Using that $N$ is
  augmented over $U^{\cJ}$, we define $\gamma(N)'$ by the
  factorization $\xymatrix@1{\gamma(N)\;\ar@{>->}[r]^{\sim}&
    \;\gamma(N)'\; \ar@{->>}[r] & \;\gamma(U^{\cJ})}$ and observe that
  since $\gamma(U^{\cJ})$ is contractible, the fiber
  $\gamma(N)^{\fib}$ of $\gamma(N)' \to \gamma(U^{\cJ})$ is a possible
  choice for a fibrant replacement of $\gamma(N)$. We obtain a
  commutative diagram
  \[\xymatrix@-1pc{
    \gamma(M)\times\gamma(N)^{\fib} \ar[r]^-{\sim}\ar[d] &
    \gamma(M)\times\gamma(N)' \ar[d]
    & \gamma(M)\times\gamma(N) \ar[d] \ar[l]_-{\sim}  \ar[r]^{\sim}& \gamma(M\boxtimes N) \ar[d]\\
    \gamma(M) \ar[r]^-{\sim} & \gamma(M) \times \gamma(U^{\cJ})
    \ar[r]^{=} & \gamma(M) \times \gamma(U^{\cJ}) \ar[r]^-{\sim} &
    \gamma(M). }\] The resulting equivalence between
  $\gamma(M)\times\gamma(N)^{\fib}$ and $\gamma(M\boxtimes N)$ lies
  over $\gamma(M)$ because the bottom composition is the identity.  It
  is easy to see that we get maps under $\gamma(M)$ and hence under
  $\gamma(P)$. The claim follows using the $\cJ$-equivalence $M\boxtimes
  (1+X)^{\cJ} \to (M+X)^{\cJ}$.
\end{proof}
Writing $\gamma(M)/\gamma(P)$ for a cofibrant model of the homotopy
cofiber of the map of $\Gamma$-spaces $\gamma(P)\to\gamma(M)$, we can
now use the relation between commutative $\cJ$-space monoids and
$\Gamma$-spaces established in~\cite[\S 7]{Sagave_spectra-of-units} to
express commutative $\cJ$-space monoid derivations by mapping spaces
in $\GammaS$:
\begin{lemma}\label{lem:Map-CSJ-GammaS} There is a natural weak
  equivalence between the mapping spaces
  \[\Map^P_M(M,(M+X)^{\cJ})\quad \text{and} \quad
  \Map_{\GammaS}\left(\gamma(M)/\gamma(P),(\gamma(1+X)^{\cJ})^{\fib}\right).\]
\end{lemma}
\begin{proof}
  We denote the mapping spaces in the comma category
  $\gamma(P)\downarrow \GammaS \downarrow \gamma(M)$ by
  $\Map^{\gamma(P)}_{\gamma(M)}(-,-)$, and we let $\xymatrix@1{\gamma(P) \;\ar@{>->}[r]&
  \;\gamma(M)^{\cof}\; \ar@{->>}[r]^{\sim}& \;\gamma(M)}$ be a cofibrant replacement of
  $\gamma(M)$ in this comma category.

  Since Quillen equivalences induce weak equivalences between the
  homotopy types of mapping spaces in the respective categories, the
  Quillen equivalences of~\cite[Corollary
  7.12]{Sagave_spectra-of-units} and
  Lemma~\ref{lem:gamma-M-plus-X-J-as-product} show that the space
  $\Map^P_M(M,(M+X)^{\cJ})$ is weakly equivalent to the mapping space
  $\Map^{\gamma(P)}_{\gamma(M)}\left(\gamma(M)^{\cof},\gamma(M)\times
    (\gamma(1+X)^{\cJ})^{\fib}\right)$ in the model category
  $(\GammaS)_{\pre}/\bof{\cJ}$ defined in~\cite[\S
  7]{Sagave_spectra-of-units}.  The iterated pullbacks that are used
  to define mapping spaces in comma categories show that the mapping
  spaces in $\left((\GammaS)_{\pre}/\bof{\cJ})\right)\downarrow
  \gamma(M)$ coincide with the mapping spaces in the category
  $(\GammaS)_{\pre}\downarrow \gamma(M)$. A similar argument shows
  that the space of maps into
  $\gamma(M)\times(\gamma(1+X)^{\cJ})^{\fib}$ in the category
  $\gamma(P)\downarrow(\GammaS)_{\pre}\downarrow \gamma(M)$ coincides
  with the space of maps into $(\gamma(1+X)^{\cJ})^{\fib} $ in the
  category $\gamma(P)\downarrow(\GammaS)_{\pre}$. Since $(1+X)^{\cJ}$
  is grouplike, the $\Gamma$-space $\gamma(1+X)^{\cJ}$ is very special
  and its level fibrant replacement $(\gamma(1+X)^{\cJ})^{\fib} $ is
  fibrant in the stable model structure on $\Gamma$-spaces. Since the
  latter model structure can be constructed as a left Bousfield
  localization of $(\GammaS)_{\pre}$, spaces of maps into
  $(\gamma(1+X)^{\cJ})^{\fib}$ coincide in these two model
  structures. So we have built a weak equivalence
  \[\Map^{\gamma(P)}_{\gamma(M)}\left(\gamma(M)^{\cof},\gamma(M)\times
    (\gamma(1+X)^{\cJ})^{\fib}\right) \simeq
  \Map^{\gamma(P)}\left(\gamma(M)^{\cof},(\gamma(1+X)^{\cJ})^{\fib}\right)\]
  Applying the functor $\Map_{\GammaS}(-,(\gamma(1+X)^{\cJ})^{\fib})$
  to the pushout
  \[\xymatrix@-1pc{
    \gamma(P) \ar@{>->}[r] \ar[d] & \gamma(M)^{\cof} \ar[d] \\
    \ast \ar[r] & \gamma(M)/\gamma(P) }\] shows that the space
  $\Map^{\gamma(P)}\left(\gamma(M)^{\cof},(\gamma(1+X)^{\cJ})^{\fib}\right)
  $ is weakly equivalent to
  $\Map_{\GammaS}\left(\gamma(M)/\gamma(P),(\gamma(1+X)^{\cJ})^{\fib}\right)$.
\end{proof}
Using Lemma~\ref{lem:Map-CSJ-GammaS} and
Proposition~\ref{prop:identification_1plusX-J}, we can prove the
remaining proposition:
\begin{proof}[Proof of Proposition~\ref{prop:CSJ-derivations}] By
  Lemma~\ref{lem:Map-CSJ-GammaS}, the space of commutative
  \mbox{$\cJ$-space} de\-rivations is weakly equivalent to
  $\Map_{\GammaS}\left(\gamma(M)/\gamma(P),(\gamma(1+X)^{\cJ})^{\fib}\right)$. We
  have shown in Proposition~\ref{prop:identification_1plusX-J} that
  the $\Gamma$-space $\gamma(1+X)^{\cJ}$ models the connective cover
  of the underlying spectrum of $X$. Keeping the notation
  $\gamma(M)/\gamma(P)$ for the (symmetric) spectrum associated with
  $\gamma(M)/\gamma(P)$~\cite{MMSS}, this spectrum being connective implies that the
  mapping space in question is equivalent to
  $\Map_{\Spsym}\left(\gamma(M)/\gamma(P),X\right)$.  The desired
  equivalence to $\Map_{A\text{-}\Mod}(A\sm (\gamma(M)/\gamma(P)), X)$
  follows by extension of scalars along $\mS \to A$.
\end{proof}


\begin{thebibliography}{10}

\bibitem{Ausoni_THH_ku}
C.~Ausoni.
\newblock Topological {H}ochschild homology of connective complex {$K$}-theory.
\newblock {\em Amer. J. Math.}, 127(6):1261--1313, 2005.

\bibitem{Ausoni-R_fraction-field}
C.~Ausoni and J.~Rognes.
\newblock Algebraic {$K$}-theory of the fraction field of topological
  {$K$}-theory.
\newblock \arxivlink{0911.4781}.

\bibitem{Baker-R_numerical}
A.~Baker and B.~Richter.
\newblock On the {$\Gamma$}-cohomology of rings of numerical polynomials and
  {$E_\infty$} structures on {$K$}-theory.
\newblock {\em Comment. Math. Helv.}, 80(4):691--723, 2005.

\bibitem{Baker-R_uniqueness}
A.~Baker and B.~Richter.
\newblock Uniqueness of {$E_\infty$} structures for connective covers.
\newblock {\em Proc. Amer. Math. Soc.}, 136(2):707--714 (electronic), 2008.

\bibitem{Basterra_Andre-Quillen}
M.~Basterra.
\newblock Andr\'e-{Q}uillen cohomology of commutative {$S$}-algebras.
\newblock {\em J. Pure Appl. Algebra}, 144(2):111--143, 1999.

\bibitem{Blumberg-M_loc-sequenceTHH}
A.~J. {Blumberg} and M.~A. {Mandell}.
\newblock {Localization for THH(ku) and the topological Hochschild and cyclic
  homology of Waldhausen categories}.
\newblock \arxivlink{1111.4003}.

\bibitem{EKMM}
A.~D. Elmendorf, I.~Kriz, M.~A. Mandell, and J.~P. May.
\newblock {\em Rings, modules, and algebras in stable homotopy theory},
  volume~47 of {\em Mathematical Surveys and Monographs}.
\newblock American Mathematical Society, Providence, RI, 1997.
\newblock With an appendix by M. Cole.

\bibitem{Hesselholt-M_local_fields}
L.~Hesselholt and I.~Madsen.
\newblock On the {$K$}-theory of local fields.
\newblock {\em Ann. of Math. (2)}, 158(1):1--113, 2003.

\bibitem{Hirschhorn_model}
P.~S. Hirschhorn.
\newblock {\em Model categories and their localizations}, volume~99 of {\em
  Mathematical Surveys and Monographs}.
\newblock American Mathematical Society, Providence, RI, 2003.

\bibitem{HSS}
M.~Hovey, B.~Shipley, and J.~Smith.
\newblock Symmetric spectra.
\newblock {\em J. Amer. Math. Soc.}, 13(1):149--208, 2000.

\bibitem{Kato_logarithmic-structures}
K.~Kato.
\newblock Logarithmic structures of {F}ontaine-{I}llusie.
\newblock In {\em Algebraic analysis, geometry, and number theory ({B}altimore,
  {MD}, 1988)}, pages 191--224. Johns Hopkins Univ. Press, Baltimore, MD, 1989.

\bibitem{MMSS}
M.~A. Mandell, J.~P. May, S.~Schwede, and B.~Shipley.
\newblock Model categories of diagram spectra.
\newblock {\em Proc. London Math. Soc. (3)}, 82(2):441--512, 2001.

\bibitem{Ogus_log-book}
A.~Ogus.
\newblock Lectures on logarithmic algebraic geometry.
\newblock Preprint.

\bibitem{Rognes_TLS}
J.~Rognes.
\newblock Topological logarithmic structures.
\newblock In {\em New topological contexts for {G}alois theory and algebraic
  geometry ({BIRS} 2008)}, volume~16 of {\em Geom. Topol. Monogr.}, pages
  401--544. Geom. Topol. Publ., Coventry, 2009.

\bibitem{Rognes-S-S_logTHH}
J.~Rognes, S.~Sagave, and C.~Schlichtkrull.
\newblock Logarithmic topological {H}ochschild homology.
\newblock In preparation.

\bibitem{Sagave_spectra-of-units}
S.~Sagave.
\newblock Spectra of units for periodic ring spectra.
\newblock \arxivlink{1111.6731}.

\bibitem{Sagave-S_diagram}
S.~Sagave and C.~Schlichtkrull.
\newblock Diagram spaces and symmetric spectra.
\newblock {\em Adv. Math.}, 231(3-4):2116--2193, 2012.

\bibitem{Sagave-S_group-compl}
S.~Sagave and C.~Schlichtkrull.
\newblock Group completion and units in $\mathcal{I}$-spaces.
\newblock {\em Algebr. Geom. Topol.}, 13(2):625--686, 2013.

\bibitem{Schlichtkrull_units}
C.~Schlichtkrull.
\newblock Units of ring spectra and their traces in algebraic {$K$}-theory.
\newblock {\em Geom. Topol.}, 8:645--673 (electronic), 2004.

\bibitem{Schwede_SymSp}
S.~Schwede.
\newblock Symmetric spectra.
\newblock Book project, available at the author's home page.

\bibitem{Schwede_S-modules-Spsym}
S.~Schwede.
\newblock {$S$}-modules and symmetric spectra.
\newblock {\em Math. Ann.}, 319(3):517--532, 2001.

\bibitem{Segal_categories}
G.~Segal.
\newblock Categories and cohomology theories.
\newblock {\em Topology}, 13:293--312, 1974.

\end{thebibliography}
\end{document}